\documentclass[10pt]{article}
\pdfoutput=1
\usepackage{geometry, amsmath, amsthm, latexsym, amssymb, graphicx, mathtools, tikz, enumerate, dsfont}
\usepackage{tikz-cd}
\usepackage[utf8]{inputenc}
\usepackage[T1]{fontenc}
\usepackage[shortlabels]{enumitem}
\setenumerate{topsep=0pt}

\usepackage{tkz-graph}
\usepackage{graphicx}
\usepackage{csquotes}
\newcommand{\git}{\mathbin{
  \mathchoice{/\mkern-6mu/}% \displaystyle
    {/\mkern-6mu/}% \textstyle
    {/\mkern-5mu/}% \scriptstyle
    {/\mkern-5mu/}}}% \scriptscriptstyle
\usepackage{lmodern}
\usepackage{arydshln}
\usepackage[shortlabels]{enumitem}
\usepackage{multirow}
\usepackage{mathdots}
\usepackage{hyperref}
\usepackage{authblk}

\theoremstyle{definition}
\newtheorem{defi}{Definition}[section]
\newtheorem{ex}[defi]{Example}
\newtheorem{obs}[defi]{Remark}

\theoremstyle{plain}
\newtheorem{theorem}[defi]{Theorem}
\newtheorem{prop}[defi]{Proposition}
\newtheorem{lemma}[defi]{Lemma}

\newtheorem*{idea*}{Idea}

\DeclareMathOperator{\SL}{SL}
\DeclareMathOperator{\GL}{GL}
\DeclareMathOperator{\PGL}{PGL}

\DeclareMathOperator{\So}{SO}
\DeclareMathOperator{\PSo}{PSO}

\DeclareMathOperator{\id}{id}
\DeclareMathOperator{\Char}{char}
\DeclareMathOperator{\ord}{ord}

\DeclareMathOperator{\Aut}{Aut}
\DeclareMathOperator{\Isom}{Isom}

\DeclareMathOperator{\Hom}{Hom}

\DeclareMathOperator{\Ad}{Ad}
\DeclareMathOperator{\Sym}{Sym}
\DeclareMathOperator{\Spec}{Spec}
\DeclareMathOperator{\Ht}{ht}
\DeclareMathOperator{\ad}{ad}
\DeclareMathOperator{\diag}{diag}
\DeclareMathOperator{\Stab}{Stab}
\DeclareMathOperator{\Lie}{Lie}
\DeclareMathOperator{\End}{End}
\DeclareMathOperator{\vol}{vol}
\DeclareMathOperator{\Sel}{Sel}

\newcommand{\lp}{\left(}
\newcommand{\rp}{\right)}

\newcommand{\la}{\langle}
\newcommand{\ra}{\rangle}
\newcommand{\Z}{\mathbb{Z}}
\newcommand{\Q}{\mathbb{Q}}
\newcommand{\R}{\mathbb{R}}
\newcommand{\C}{\mathbb{C}}

\newcommand{\G}{\mathbb{G}}
\newcommand{\F}{\mathbb{F}}

\newcommand{\Gg}{\mathfrak{g}}
\newcommand{\hh}{\mathfrak{h}}
\newcommand{\fz}{\mathfrak{z}}
\newcommand{\fc}{\mathfrak{c}}
\newcommand{\ft}{\mathfrak{t}}
\newcommand{\fsl}{\mathfrak{sl}}

\newcommand{\cA}{\mathcal{A}}
\newcommand{\cR}{\mathcal{R}}

\newcommand{\cW}{\mathcal{W}}
\newcommand{\cS}{\mathcal{S}}
\newcommand{\cF}{\mathcal{F}}
\newcommand{\cJ}{\mathcal{J}}
\newcommand{\cB}{\mathcal{B}}
\newcommand{\cU}{\mathcal{U}}
\newcommand{\eps}{\varepsilon}
\newcommand{\ul}{\underline}
\newcommand{\ol}{\overline}

\usepackage[
backend=biber,
style=alphabetic,
sorting=nyt
]{biblatex}
\title{The density of ADE families of curves having squarefree discriminant}
\author{Martí Oller}
\affil{\small University of Cambridge, Department of Pure Mathematics and Mathematical Statistics, Centre for Mathematical Sciences, Wilberforce Road,Cambridge CB3 0WB \\email: mo512@cam.ac.uk}

\begin{document}

\maketitle
\begin{abstract}
We determine the density of curves having squarefree discriminant in some
families of curves that arise from Vinberg representations, showing that
the global density is the product of the local densities. We do so using
the framework of Thorne and Laga's PhD theses and Bhargava's orbit-counting 
techniques. This paper generalises a previous result by Bhargava, Shankar
and Wang.
\end{abstract}
\tableofcontents
\section{Introduction}

The aim of this paper is to determine the density of curves in certain families
that have squarefree discriminant. We do so following the techniques
in arithmetic statistics developed by Bhargava and his collaborators. 
The main idea is that many arithmetic objects of interest can be 
parametrised by the rational or integral orbits of a certain 
representation $(G,V)$: in this situation, Bhargava's geometry-of-numbers
methods allow to count these integral orbits of $V$, which consequently provides
information on the desired arithmetic objects that would be otherwise
difficult to obtain. This idea has led to many impressive results in 
number theory; see \cite{BOverview} or \cite{Ho2013} for an overview.

The present paper is inspired by the recent paper \cite{BSWsquarefree} by 
Bhargava, Shankar and Wang, in which they compute the density of monic integral
polynomials of a given degree that have squarefree discriminant. The main
technical difficulty is to bound the tail estimate of polynomials
having discriminant ``weakly divisible'' by the square a large prime (this notion will be defined later). 
They do so using the representation of $G = \So_n$ on
the space $V$ of $n\times n$ symmetric matrices. By relating polynomials
with discriminant divisible by $p^2$ for a large $p$ to certain integral
orbits of the representation $(G,V)$, they get the desired result using
the aforementioned geometry-of-numbers techniques. Similar methods were
used in \cite{BSW-2} in the non-monic case with a
different representation, and also in \cite{BHo} for certain families
of elliptic curves (in particular, their $F_2$ case essentially corresponds
to our $D_4$ case).

A key observation, which motivates our results, is that the representation
studied in \cite{BSWsquarefree} arises as a particular case of the more 
general families of representations studied in \cite{ThorneThesis}.
Using the framework of Vinberg theory, Thorne found that given a simply
laced Dynkin diagram, we can naturally associate to it a 
family of curves and a coregular representation $(G,V)$, where the rational
orbits of the representation are related to the arithmetic of the curves
in the family. These results have been used, implicitly and explicitly,
to study the size of $2$-Selmer groups of the Jacobians of these curves, 
see \cite{BG,SW,ShankarD2n+1,ThorneE6,RTE78,LagaE6} for some particular 
cases. Later, Laga unified, reproved and extended all these results in \cite{LagaThesis} 
in a uniform way.

Our aim is to compute the density of curves having squarefree discriminant
in these families of $ADE$ curves. We will do so by reinterpreting the
methods in \cite{BSWsquarefree} in the language of \cite{ThorneThesis}
and \cite{LagaThesis}. As a corollary, we will obtain the asymptotics
for the number of integral reducible orbits of these representations,
following \cite{SSSVcusp}.

Let $\mathcal{D}$ be a Dynkin diagram of type $A,D,E$. In Section \ref{subs:Vinberg}, 
we will construct a representation $(G,V)$ associated to $\mathcal{D}$,
and in Section \ref{subs:Transverse} we will construct a family of curves
$C \to B$. Here, $B$ is isomorphic to the Geometric Invariant Theory (GIT)
quotient $V \git G := \Spec \Q[V]^G$. We see that $B$ can be identified
with an affine space, and we write $B = \Spec \Q[p_{d_1},\dots,p_{d_r}]$.
Given $b \in B$, we define its \emph{height} to be
\[
\text{ht}(b) := \sup\left(|p_{d_1}(b)|^{1/d_1},\dots,|p_{d_r}(b)|^{1/d_r}\right).
\]
Denote by $C_b$ the preimage of a given $b\in B$ under the map $C \to B$;
it will be a curve of the form given by Table \ref{table:curves}. The main
result of this paper concerns the density of squarefree values of the
discriminant $\Delta(C_b)$ of the curve (or equivalently, the discriminant
$\Delta(b)$ defined in Section \ref{subs:Vinberg}). A definition for the
discriminant of a plane curve can be found in \cite[\S 2]{Sutherland2019}, for instance.
We remark that in our definition of discriminant, we assume that it is
an integer-valued polynomial in multiple variables, normalised so that
the coefficients have greatest common divisor 1 (for instance, the usual
discriminant for elliptic curves contains a factor of $16$: we omit it
in our case).

Our result is related to the $p$-adic density of these squarefree values: 
we will denote by $\rho(\mathcal{D}_p)$ the $p$-adic density of curves in 
the family $C \to B$ having discriminant indivisible by $p^2$ in $\Z_p$;
this is obtained by taking all the (finitely many) elements in $b \in B(\Z/p^2\Z)$
and counting the proportion of them that have non-zero discriminant in
$\Z/p^2\Z$. We note that under our assumptions on the discriminant, none 
of the local densities vanish; this can be checked with a case-by-case computation.
\begin{theorem}
\label{theo:main}
We have
\[
\lim_{X \to \infty} \frac{\#\{b \in B(\Z) \mid \Delta(b) \text{ is squarefree, } \Ht(b) < X\}}{\#\{b \in B(\Z) \mid \Ht(b) < X\}} = \prod_{p} \rho(\mathcal{D}_p).
\]
\end{theorem}
To prove this theorem, we need to obtain a tail estimate to show that
not too many $b \in B(\Z)$ have discriminant divisible by $m^2$ for large squarefree integers $m$.
A key observation in \cite{BSWsquarefree} is to separate those $b \in B(\Z)$
with $p^2 | \Delta(b)$ for a prime $p$ in two separate cases:
\begin{enumerate}
\item If $p^2 | \Delta(b+pc)$ for all $c \in B(\Z)$, we say $p^2$ \emph{strongly divides}
$\Delta(b)$ (in other words, $p^2$ divides $\Delta(b)$ for ``mod $p$ reasons'').
\item If there exists $c \in B(\Z)$ such that $p^2 \nmid \Delta(b+pc)$,
we say $p^2$ \emph{weakly divides} $\Delta(b)$ (in other words, $p^2$ 
divides $\Delta(b)$ for ``mod $p^2$ reasons'').
\end{enumerate}
Similarly, for a squarefree number $m$, we say that $m^2$ strongly (resp. weakly)
divides $\Delta(b)$ if $p^2$ strongly (resp. weakly) divides $\Delta(b)$
for all primes $p$ dividing $m$.
We let $\cW_{m}^{(1)},\cW_{m}^{(2)}$ denote the set 
of $b \in B(\Z)$ whose discriminant is strongly (resp. weakly) divisible 
by $m^2$. We prove tail estimates for these two sets separately. The 
argument in the weakly divisible case will require our squarefree integers $m$ to 
avoid a finite number of primes: in Section \ref{subs:integral}, we will
consider an integer $N$ which contains all ``bad primes''.
\begin{theorem}
\label{theorem:tail}
There exists a constant $\delta > 0$ such that for any positive real number $M$ we have:
\begin{align*}
\sum_{\substack{m > M \\ m \text{ squarefree}}} \{b \in \cW_{m}^{(1)} \mid \Ht(b) < X\} &= O_{\eps}\left(\frac{X^{\dim V+\eps}}{M}\right) + O_{\eps}\left(X^{\dim V - 1+\eps}\right),\\
\sum_{\substack{m > M \\ m \text{ squarefree} \\ (m,N) = 1}} \{b \in \cW_{m}^{(2)} \mid \Ht(b) < X\} &= O_{\eps}\left(\frac{X^{\dim V+\eps}}{M}\right) + O\left(X^{\dim V - \delta}\right).
\end{align*}
The implied constants are independent of $X$ and $M$.
\end{theorem}

As in \cite[Theorem 1.5(a)]{BSWsquarefree}, the strongly divisible case
follows from the use of the Ekedahl sieve; more precisely, it follows from the
results in \cite[Theorem 3.5, Lemma 3.6]{BEkedahl} and the fact that
the discriminant polynomial is irreducible by \cite[Lemma 4.2]{LagaThesis}.
Therefore, it remains to prove the (substantially harder) weakly divisible
case, which is the content of most of this paper.

We start in Section \ref{section:Prelim} by giving the necessary background
and introducing our objects of interest, most importantly the representation
$(G,V)$ coming from Vinberg theory and the associated family of curves
$C \to B$. The main step in the proof of Theorem \ref{theorem:tail} is
done in Section \ref{section:construct}, where given a $b \in \cW_m^{(2)}$
we obtain a special integral $G(\Z)$-orbit in $V$ whose elements have invariants
$b$. We additionally consider a distinguished subspace $W_0(\Z) \subset V(\Z)$, 
and we define a $Q$-invariant for the elements of $W_0(\Z)$. Then, we will
see that the elements in the constructed orbit have large $Q$-invariant
when they intersect $W_0(\Z)$ (which happens always except for a negligible
amount of times by cutting-off-the-cusp arguments). This construction
is the analogue of \cite[\S 2.2, \S 3.2]{BSWsquarefree}; we give a more detailed
comparison at the end of Section \ref{section:construct}.

In view of all that, to prove Theorem \ref{theorem:tail} it suffices to
bound the number of these distinguished $G(\Z)$-orbits in $W_0(\Z)$ having
large $Q$-invariant. However, before doing that, we will need to take
a small detour and estimate the number of \emph{all} reducible
$G(\Z)$-orbits in $V(\Z)$. A $G(\C)$-orbit in $V(\C)$ can split into
multiple $G(\Q)$-orbits, and among these $G(\Q)$-orbits there is a
``distinguished'' one (namely, the one given by the Kostant section, as
defined in Section \ref{subs:Transverse}). We say that an element in $V(\Q)$ is
\emph{reducible} if it falls into this special $G(\Q)$-orbit.
Using Bhargava's geometry-of-numbers arguments, and in particular the techinques in
the cusp developed in \cite{SSSVcusp}, we will obtain the following result:
\begin{theorem}
The number $N(V(\Z)^{red},X)$ of reducible $G(\Z)$-orbits on $V(\Z)$ 
of height at most $X$ is
\[
N(V(\Z)^{red},X) = CX^{\dim V} + O(X^{\dim V - \delta}),
\]
where $C,\delta$ are real positive constants. The constant $C$ will be
explicitly determined in Section \ref{subs:constant}.
\end{theorem}
The proof of this theorem relies on the construction of a \emph{box-shaped}
fundamental domain for the action of $G(\Z)$ on $G(\R)$, which will be
carried out in Section \ref{subs:FDomains}. Following that, we will mostly
follow the steps in \cite[\S 4]{SSSVcusp}, relying on critical reductions
given by \cite[\S 8]{LagaThesis}; and we will conclude our proof by using
elementary but lengthy case-by-case computations. We remark that our proof implicitly also relies
on other implicit case-by-case computations: namely, the cutting-off-the-cusp
result in Proposition \ref{prop:cusp} relies on an (even more tedious)
exhaustive analysis of all cases, sometimes relying on lengthy computations
on a computer (cf. \cite[Proposition 4.5]{RTE78}).

In Section \ref{section:final}, we will conclude the proof of Theorem
\ref{theorem:tail}, from which Theorem \ref{theo:main} will follow using
a squarefree sieve.
The sieve is carried out in a general enough setting that allows us
to count the density of subsets in $B(\Z)$ defined by infinitely many
congruence conditions. In particular, we get an application of our result
to the context of \cite{LagaThesis}, which allows us to get an upper
bound on the average size of 2-Selmer groups of families defined by infinitely
many congruence condtions. For $b \in B(\Z)$, denote by $J_b$ the Jacobian
of the curve $C_b$.
\begin{theorem}
Let $m$ be the number of marked points of the family $C \to B$, as given 
in Table \ref{table:curves}. Let $\cS$ be a $\kappa$-acceptable subset 
of $B(\Z)$ in the sense of Section \ref{subs:Sieve}. Then, we have
\[
\limsup_{X \to \infty} \frac{\sum_{b \in \cS,\,\Ht(b)<X}\#\Sel_2 J_b}{\#\{b \in \cS \mid \Ht(b) < X\}} \leq 3\cdot2^{m-1}.
\]
\end{theorem}

\paragraph*{Acknowledgements.}
This paper was written while the author was a PhD student under the supervision
of Jack Thorne. I would like to thank him for providing many useful
suggestions, guidance and encouragement during the process, and for revising
an early version of this manuscript. I also wish to thank Jef Laga and 
the anonymous referees for their helpful comments. The project 
that gave rise to these results received the support of a fellowship 
from ``la Caixa'' Foundation (ID 100010434). The fellowship code is LCF/BQ/EU21/11890111.
The author wishes to thank them, as well as the Cambridge Trust and the DPMMS,
for their support.
\section{Preliminaries}
\label{section:Prelim}

In this section, we introduce our representation $(G,V)$ of interest,
together with some of its basic properties. We do so mostly following 
\cite[\S 2]{ThorneThesis} and \cite[\S 3]{LagaThesis}.

\subsection{Vinberg representations}
\label{subs:Vinberg}
Let $H$ be a split adjoint simple group of type $A,D,E$ over $\Q$. We assume
$H$ is equipped with a pinning $(T,P,\{X_{\alpha}\})$, meaning:
\begin{itemize}
\item $T \subset H$ is a split maximal torus (determining a root system $\Phi_{H}$).
\item $P \subset H$ is a Borel subgroup containing $T$ (determining a root basis $S_{H} \subset \Phi_{H}$).
\item $X_{\alpha}$ is a generator for $\hh_{\alpha}$ for each $\alpha \in S_{H}$.
\end{itemize}
Let $W = N_{H}(T)/T$ be the Weyl group of $\Phi_{H}$, and let $\mathcal{D}$ be
the Dynkin diagram of $H$. Then, we have the following exact sequences:
\begin{equation}
\label{eq:pin}
\begin{tikzcd}
0 \arrow{r} & H \arrow{r} & \Aut(H) \arrow{r} & \Aut(\mathcal{D}) \arrow{r} & 0
\end{tikzcd}
\end{equation}
\begin{equation}
\label{eq:ex2}
\begin{tikzcd}
0 \arrow{r} & W \arrow{r} & \Aut(\Phi_{H}) \arrow{r} & \Aut(\mathcal{D}) \arrow{r} & 0
\end{tikzcd}
\end{equation}
The subgroup $(T,P,\{X_{\alpha}\}) \subset \Aut(H)$ of automorphisms
of $H$ preserving the pinning determines a
splitting of \eqref{eq:pin}. Then, we can define $\vartheta \in \Aut(H)$
as the unique element in $(T,P,\{X_{\alpha}\})$ such that its image
in $\Aut(\mathcal{D})$ under \eqref{eq:pin} coincides with the image of $-1 \in \Aut(\Phi_{H})$
under \eqref{eq:ex2}. Writing $\check{\rho}$ for the sum of fundamental
coweights with respect to $S_{H}$, we define
\[
\theta := \vartheta \circ \Ad(\check{\rho}(-1)) = \Ad(\check{\rho}(-1)) \circ \vartheta.
\]
The map $\theta$ defines an involution of $H$, and so $d\theta$ defines an
involution of the Lie algebra $\hh$. By considering $\pm 1$ eigenspaces, we obtain a $\Z/2\Z$-grading
\[
\hh = \hh(0) \oplus \hh(1),
\]
where $[\hh(i),\hh(j)] \subset \hh(i+j)$. We define $G = (H^{\theta})^{\circ}$
and $V = \hh(1)$, which means that $V$ is a representation of $G$ by restriction
of the adjoint representation. Moreover, we have $\text{Lie}(G) = \hh(0)$.

We have the following basic result \cite[Theorem 1.1]{Panyushev} on the
GIT quotient $B := V \git G = \Spec \Q[V]^G$.

\begin{theorem}
Let $\mathfrak{c} \subset V$ be a Cartan subspace. Then, $\mathfrak{c}$
is a Cartan subalgebra of $\hh$, and the map $N_G(\mathfrak{c}) \to W_{\fc} := N_H(\fc)/Z_H(\fc)$
is surjective. Therefore, the canonical inclusions $\fc \subset V \subset \hh$
induce isomorphisms
\[
\fc \git W_{\fc} \cong V \git G \cong \hh \git H.
\]
In particular, all these quotients are isomorphic to a finite-dimensional affine space.
\end{theorem}

For any field $k$ of characteristic zero, we can define 
the \emph{discriminant polynomial} $\Delta \in k[\hh]^H$ as the image of 
$\prod_{\alpha \in \Phi_H} \alpha$ under the isomorphism $k[\mathfrak{t}]^{W} \xrightarrow{\sim} k[\hh]^H$.
The discriminant can also be regarded as a polynomial in $k[B]$ through the isomorphism
$k[\hh]^H \cong k[V]^G = k[B]$.
We can relate the discriminant to one-parameter subgroups, which we now introduce. If $k/\Q$
is a field and $\lambda \colon \G_m \to G_{k}$ is a homomorphism, there
exists a decomposition $V = \sum_{i \in \Z} V_i$, where $V_i := \{v \in V(k) \mid \lambda(t)v = t^iv \; \forall t \in \G_m(k)\}$.
Every vector $v \in V(k)$ can be written as $v = \sum v_i$, where $v_i \in V_i$;
we call the integers $i$ with $v_i \neq 0$ the \emph{weights} of $v$.
Finally, we recall that an element $v \in \hh$ is \emph{regular} if
its centraliser has minimal dimension.
\begin{prop}
\label{prop:H-M}
Let $k/\Q$ be a field, and let $v \in V(k)$. The following are equivalent:
\begin{enumerate}
\item $v$ is regular semisimple.
\item $\Delta(v) \neq 0$.
\item For every non-trivial homomorphism $\lambda \colon \G_m \to G_{k^s}$,
$v$ has a positive weight with respect to $\lambda$.
\end{enumerate}
\end{prop}
\begin{proof}
The reasoning is the same as in \cite[Corollary 2.4]{RTE78}.
\end{proof}

We remark that the Vinberg representation $(G,V)$ can be identified explictly. For the
reader's convenience, we reproduce the explicit description written in \cite[\S 3.2]{LagaThesis}
in Table \ref{table:explicit}. We refer the reader to \emph{loc. cit.} for the
precise meaning of some of these symbols.

\begin{table}[h]
\begin{center}
\begin{tabular}{l|c|c}
Type & $G$ & $V$ \\
\hline
$A_{2n}$ & $\So_{2n+1}$ & $\Sym^2(2n+1)_0$ \\
$A_{2n+1}$ & $\PSo_{2n+2}$ & $\Sym^2(2n+2)_0$ \\
$D_{2n}$ ($n \geq 2$) & $\So_{2n} \times \So_{2n} / \Delta(\mu_2)$ & $2n \boxtimes 2n$ \\
$D_{2n+1}$ ($n \geq 2$) &  $\So_{2n+1} \times \So_{2n+1}$ & $(2n+1)\boxtimes (2n+1)$ \\
$E_{6}$ & PSp$_8$ & $\wedge_{0}^4 8$ \\
$E_{7}$ & $\SL_{8}/\mu_4$ & $\wedge^4 8$ \\
$E_{8}$ & Spin$_{16}/\mu_2$& half spin\\
\end{tabular}
\caption{Explicit description of each representation}
\label{table:explicit}
\end{center}
\end{table}

\subsection{Restricted roots}
\label{section:ResRoots}

In the previous section, we considered the root system $\Phi := \Phi_H$ of
$H$, but we will also need to understand the restricted root system
$\Phi(G,T^{\theta})$ and the set of weights $\Phi_V$ of the action of
$T^{\theta}$ on $V$. This will be particularly important when defining
the distinguished subspace $W_0 \subset V$ and the $Q$-invariant in Section
\ref{section:construct}. The exposition in this section is based on \cite[\S 2.3]{ThorneE6}.

Write $\Phi/\vartheta$ for the orbits of $\vartheta$ on $\Phi$, where
$\vartheta$ is the pinned automorphism defined in the previous section.
\begin{lemma}
\begin{enumerate}
\item The map $X^{\ast}(T) \to X^{\ast}(T^{\theta})$ is surjective, and
the group $G$ is adjoint. In particular, $X^{\ast}(T^{\theta})$ is spanned
by $\Phi(G,T^{\theta})$.
\item Let $\alpha, \beta \in \Phi$. Then, the image of $\alpha$ in $X^{\ast}(T^{\theta})$
is non-zero, and $\alpha,\beta$ have the same image if and only if either
$\alpha = \beta$ or $\alpha = \vartheta(\beta)$.
\end{enumerate}
\end{lemma}
\begin{proof}
This is \cite[Lemma 2.5]{ThorneE6}.
\end{proof}
Hence, we can identify $\Phi/\vartheta$ with its image in $X^{\ast}(T^{\theta})$. We
note that $\vartheta = 1$ if and only if $-1$ is an element of the Weyl
group $W(H,T)$; in this case $\Phi/\vartheta$ coincides with $\Phi$.

We can write the following decomposition:
\[
\hh = \ft \oplus \bigoplus_{a \in \Phi/\vartheta} \hh_a,
\]
with $\ft = \ft^{\theta} \oplus V_0$ and $\hh_a = \Gg_a \oplus V_a$, so that
\[
\Gg = \ft^{\theta} \bigoplus_{a \in \Phi/\vartheta} \Gg_a, \quad V = V_0 \oplus \bigoplus_{a \in \Phi/\vartheta} V_a.
\]
Given $a \in \Phi/\vartheta$, we can identify $\Gg_a$ and $V_a$ explicitly
according to the value of $s = (-1)^{\la \alpha,\check{\rho} \ra}$:
\begin{enumerate}
\item $a = \{\alpha\}$ and $s = 1$. Then, $V_a = 0$ and $\Gg_{\alpha}$ is spanned by $X_{\alpha}$.
\item $a = \{\alpha\}$ and $s = -1$. Then, $V_a$ is spanned by $X_{\alpha}$ and $\Gg_{\alpha} = 0$.
\item $a = \{\alpha,\vartheta(\alpha)\}$, with $\alpha \neq \vartheta(\alpha)$. 
Then, $V_a$ is spanned by $X_{\alpha}-sX_{\vartheta(\alpha)}$ and $\Gg_{\alpha}$ is spanned by $X_{\alpha}+sX_{\vartheta(\alpha)}$.
\end{enumerate}

We note that $\vartheta$ preserves the height of a root $\alpha$
with respect to the basis $S_H$ (recall that the height of a root $\alpha$
is defined as $\sum_{i} c_i$, where $\alpha = \sum_{\alpha_i \in S_H} c_i \alpha_i$ is the
decomposition as the sum of simple roots). Therefore, it will make sense to define
the \emph{height} of a root $a \in \Phi/\vartheta$ as the height
of any element in $\vartheta^{-1}(a)$.

\begin{obs}
It will be important for us to define the height of a root in $\Phi/\vartheta$ relative
to its corresponding height in $\hh$ and not relative to its height with
respect to some choice of basis of the root system $\Phi(G,T^{\theta})$.
As an example, consider the $E_6$ case following the conventions of
\cite{ThorneE6}. Say that a root basis for $\Phi_H$ is $\{\alpha_1,\dots,\alpha_6\}$
and a basis for $\Phi(G,T^{\theta})$ is $(a_1,a_2,a_3,a_4) = (\alpha_3+\alpha_4,\alpha_1,\alpha_3,\alpha_2+\alpha_4)$.
If a root $a \in \Phi_V$ can be expressed as $a = \sum n_i a_i = \sum m_i \alpha_i$
for some integers $n_i,m_i \in \Z$, under our definitions the height of 
a root is $\Ht(\alpha) = \sum m_i$ and not $\sum n_i$. This is different
than the natural notion of height we might arrive at if we consider the 
weights of $V$ as a representation of $G$ abstractly.
\end{obs}

\subsection{Transverse slices over $V \git G$}
\label{subs:Transverse}

In this section, we present some remarkable properties of the map
$\pi \colon V \to B$, where we recall that $B := V \git G$ is the GIT quotient.

\begin{defi}
An \emph{$\fsl_2$-triple} of $\hh$ is a triple $(e,h,f)$ of non-zero elements
of $\hh$ satisfying
\[
[h,e] = 2e,\quad [h,f] = -2f, \quad [e,f] = h.
\]
Moreover, we say this $\fsl_2$-triple is \emph{normal} if $e,f \in \hh(1)$ and $h \in \hh(0)$.
\end{defi}
\begin{theorem}[Graded Jacobson-Morozov]
\label{theo:J-M}
Every non-zero nilpotent element $e \in \hh(1)$ is contained in a normal $\fsl_2$-triple.
If $e$ is also regular, then it is contained in a unique normal $\fsl_2$-triple.
\end{theorem}
\begin{proof}
The first part of the statement is \cite[Lemma 2.17]{ThorneThesis}, and
the second part follows from \cite[Lemma 2.14]{ThorneThesis}.
\end{proof}
\begin{defi}
Let $r$ be the rank of $\hh$.
We say an element $x \in \hh$ is \emph{subregular} if $\dim \fz_{\hh}(x) = r + 2$.
\end{defi}
Subregular nilpotent elements in $V$ exist by \cite[Proposition 2.27]{ThorneThesis}.
Let $e \in V$ be such an element, and fix a normal $\fsl_2$-triple $(e,h,f)$
using Theorem \ref{theo:J-M}.
Let $C = e + \fz_{V}(f)$, and consider the natural morphism $\varphi \colon C \to B$.

\begin{theorem}
\label{theo:curves}
\begin{enumerate}
\item The geometric fibres of $\varphi$ are reduced connected curves.
For $b \in B(k)$, the corresponding curve $C_b$ is smooth if and only if
$\Delta(b) \neq 0$.
\item The central fibre $\varphi^{-1}(0)$ has a unique singular point which is a simple singularity
of type $A_n,D_n,E_n$, coinciding with the type of $H$. 
\item We can choose coordinates $p_{d_1},\dots,p_{d_r}$ in $B$, with $p_{d_i}$ being homogeneous
of degree $d_i$, and coordinates $(x,y,p_{d_1},\dots,p_{d_r})$ on $C$ such
that $C \to B$ is given by Table \ref{table:curves}.
\end{enumerate}
\end{theorem}
\begin{proof}
See \cite[Theorem 3.8]{ThorneThesis}.
\end{proof}

\begin{table}
\begin{center}
\begin{tabular}{l|l|c}
Type & Curve & \# Marked points\\
\hline
$A_{2n}$ & $y^2 = x^{2n+1} + p_2x^{2n-1} + \dots + p_{2n+1}$ & $1$\\
$A_{2n+1}$ & $y^2 = x^{2n+2} + p_2x^{2n} + \dots + p_{2n+1}$ & $2$ \\
$D_{2n}$ ($n \geq 2$) & $y(xy+p_{2n}) = x^{2n-1} + p_2x^{2n-2} + \dots + p_{4n-2}$ & $3$\\
$D_{2n+1}$ ($n \geq 2$) & $y(xy+p_{2n+1}) = x^{2n} + p_2x^{2n-1} + \dots + p_{4n}$ & $2$ \\
$E_{6}$ & $y^3 = x^4 + (p_2x^2 + p_5x + p_8)y + (p_6x^2 + p_9x+p_{12})$ & $1$\\
$E_{7}$ & $y^3 = x^3y + p_{10}x^2 + x(p_2y^2 + p_8y + p_{14}) + p_6y^2 + p_{12}y + p_{18}$ & $2$\\
$E_{8}$ & $y^3 = x^5 + (p_2x^3 + p_8 x^2 + p_{14}x + p_{20})y + (p_{12}x^3 + p_{18}x^2 + p_{24}x + p_{30})$ & $1$\\
\end{tabular}
\caption{Families of curves}
\label{table:curves}
\end{center}
\end{table}

Our choice of pinning in Section \ref{subs:Vinberg} determines a natural choice
of a regular nilpotent element, namely $E = \sum_{\alpha \in S_H} X_{\alpha} \in V(\Q)$.
Let $(E,H,F)$ be its associated normal $\fsl_2$-triple by Theorem \ref{theo:J-M}. We define the
affine linear subspace $\kappa_E := (E + \fz_{\hh}(F)) \cap V$ as the
\emph{Kostant section} associated to $E$. Whenever $E$ is understood, we will just
denote the Kostant section by $\kappa$.

\begin{theorem}
\label{theo:Kostant}
The composition $\kappa \hookrightarrow V \to B$ is an isomorphism, and
every element of $\kappa$ is regular.
\end{theorem}
\begin{proof}
See \cite[Lemma 3.5]{ThorneThesis}.
\end{proof}

\begin{defi}
\label{defi:reducible}
Let $k/\Q$ be a field and let $v \in V(k)$. We say $v$ is \emph{$k$-reducible}
if $\Delta(v) = 0$ or if $v$ is $G(k)$-conjugate to some Kostant section,
and \emph{$k$-irreducible} otherwise.
\end{defi}
We will typically refer to $\Q$-(ir)reducible elements simply as (ir)reducible.
We note that if $k$ is algebraically closed, then all elements of $V$
are reducible, see \cite[Proposition 2.11]{LagaThesis}.

\subsection{Integral structures}
\label{subs:integral}
So far, we have considered our objects of interest over $\Q$, but for our
purposes it will be crucial to define integral structures for $G$ and $V$.

The structure of $G$ over $\Z$ comes from the general classification of
split reductive groups over any non-empty scheme $S$: namely, every root
datum is isomorphic to the root datum of a split reductive $S$-group
(see \cite[Theorem 6.1.16]{ConradRed}). By considering the root datum
$\Phi(G,T^{\theta})$ studied in Section \ref{section:ResRoots} and the
scheme $S = \Spec \Z$, we get a split reductive group $\ul{G}$
defined over $\Z$, such that its base change to $\Q$ coincides with $G$.
By \cite[Lemma 5.1]{Richardson}, we know that $T^{\theta},P^{\theta}$ are
a maximal split torus and a Borel subgroup of $G$, respectively.
We also get integral structures for $\ul{T}^{\theta}$ and $\ul{P}^{\theta}$
inside of $\ul{G}$.
\begin{prop}
\label{prop:cl1}
$\ul{G}$ and $\ul{P}^{\theta}$ have class number 1: $\ul{G}(\mathbb{A}^{\infty}) = G(\Q)\ul{G}(\hat{\Z})$
and $\ul{P}^{\theta}(\mathbb{A}^{\infty}) = P^{\theta}(\Q)\ul{P}^{\theta}(\hat{\Z})$.
\end{prop}
\begin{proof}
Note that $\text{cl}(\ul{G}) \leq \text{cl}(\ul{P}^{\theta}) \leq \text{cl}(\ul{T}^{\theta})$ 
by \cite[Theorem 8.11, Corollary 1]{PlatonovRapinchuk}. We see that 
$\ul{T}$ has class number 1 using \cite[Theorem 8.11, Corollary 2]{PlatonovRapinchuk}, 
since $G$ contains a $\Q$-split torus consisting of diagonal matrices in 
$\GL(V)$ and $\Q$ has class number 1.
\end{proof}

To obtain the $\Z$-structure for $V$, we consider $\hh$ as a semisimple
$G$-module over $\Q$ via the restriction of the adjoint representation.
This $G$-module splits into a sum of simple $G$-modules:
\[
\hh = \left(\oplus_{i=1}^r V_i \right) \oplus \left(\oplus_{i=1}^s \Gg_i \right),
\]
where $\oplus V_i = V$ and $\oplus \Gg_i = \Gg$, since both subspaces are
$G$-invariant. For each of these irreducible representations, we can choose
highest weight vectors $v_i \in V_i$ and $w_i \in \Gg_i$, and we then consider
\[
\ul{V}_i := \text{Dist}(\ul{G})v_i, \quad \ul{\Gg}_i := \text{Dist}(\ul{G})w_i,
\]
where $\text{Dist}(\ul{G})$ the algebra of distributions of $\ul{G}$ 
(see \cite[\nopp I.7.7]{Jantzen}). By the results in 
\cite[\nopp II.8.3]{Jantzen}, we have that $V_i = \Q \otimes_{\Z} \ul{V}_i$,
$\Gg_i = \Q \otimes_{\Z} \ul{\Gg}_i$ and that $\ul{V} := \oplus \ul{V}_i$
is a $\ul{G}$-stable lattice inside $V$. By scaling the highest weight
vectors if necessary, we will assume that $E \in \ul{V}(\Z)$.

We can also consider an integral structure $\ul{B}$ on $B$. We can take
the polynomials $p_{d_1},\dots,p_{d_r}\in\Q[V]^G$ determined in Section \ref{subs:Transverse}
and rescale them using the $\G_m$-action $t \cdot p_{d_i} = t^{d_i}p_{d_i}$
to make them lie in $\Z[\ul{V}]^{\ul{G}}$. We let $\ul{B} := \Spec \Z[p_{d_1},\dots,p_{d_r}]$
and write $\pi \colon \ul{V} \to \ul{B}$ for the corresponding morphism.
We may additionally assume that the discriminant $\Delta$ defined in
Section \ref{subs:Vinberg} lies in $\Z[\ul{V}]^{\ul{G}}$, where the coefficients
of $\Delta$ in $\Z[p_{d_1},\dots,p_{d_r}]$ may be assumed to not have a
common divisor.

A crucial step in our argument will be to make our constructions in $\Z_p$
for all $p$ and then glue it all together using the class number one 
property in Proposition \ref{prop:cl1}. For this, we will need the
following lemma, which records the existence of orbits in $\ul{V}(\Z_p)$
(cf. \cite[Lemma 2.8]{ThorneE6}):

\begin{lemma}
There exists an integer $N_0 \geq 1$ such that for all primes $p$ and
for all $b \in \ul{B}(\Z_p)$ we have $N_0\cdot\kappa_b \in \ul{V}(\Z_p)$.
\end{lemma}

Our arguments in Section \ref{section:construct} will implicitly rely on integral geometric properties of
the representation $(G,V)$. In there, we will need to avoid
finitely many primes, or more precisely to work over $S = \Spec \Z[1/N]$
for a suitable $N \geq 1$. By combining the previous lemma and the 
spreading out properties in \cite[\S 7.2]{LagaThesis}, we get:
\begin{prop}
\label{prop:admissible}
There exists a positive integer $N \geq 1$ such that:
\begin{enumerate}
\item For every $b \in \ul{B}(\Z)$, the corresponding Kostant section 
$\kappa_b$ is $G(\Q)$-conjugate to an element in $\frac{1}{N}\ul{V}(\Z)$.
\item $N$ is admissible in the sense of \cite[\S 7.2]{LagaThesis}.
\end{enumerate}
\end{prop}
In particular, we will always assume that $N$ is even. We fix the integer $N$ in 
Proposition \ref{prop:admissible} throughout the rest of the paper. We will
also drop the underline notation for the objects defined over $\Z$, and
just refer to $\ul{G},\ul{V}\dots$ as $G,V\dots$ by abuse of notation.

To end this section, we consider some further integral properties
of the Kostant section. In Section \ref{subs:Transverse}, we considered 
$\kappa$ defined over $\Q$, and now we will consider some of its properties
over $\Z_p$. Consider the decomposition
\[
\hh = \bigoplus_{j \in \Z} \hh_j
\]
according to the height of the roots. If $P^{-}$ is the negative Borel
subgroup of $H$, $N^{-}$ is its unipotent radical and $\mathfrak{p}^-$
and $\mathfrak{n}^-$ are their respective Lie algebras,
we have $\mathfrak{p}^- = \bigoplus_{j \leq 0} \hh_j$, 
$\mathfrak{n}^- = \bigoplus_{j < 0} \hh_j$ and $[E,\hh_j] \subset \hh_{j+1}$. 

\begin{theorem}
\label{theo:intKostant}
Let $R$ be a ring in which $N$ is invertible. Then:
\begin{enumerate}
\item $[E,\mathfrak{n}_R^-]$ has a complement in
$\mathfrak{p}_R^-$ of rank $\text{rk}_R \, \mathfrak{p}_R^- - \text{rk}_R \, \mathfrak{n}_R^-$;
call it $\Xi$.
\item The action map $N^- \times (E + \Xi) \to E + \mathfrak{p}^-$ is
an isomorphism over $R$.
\item Both maps in the composition $E + \Xi \to (E + \mathfrak{p}^-) \git N^- \to \hh \git H$
are isomorphisms over $R$.
\end{enumerate}
\end{theorem}
\begin{proof}
See \cite[\S 2.3]{AFV}.
\end{proof}

\begin{obs}
If $R$ is a field of characteristic not dividing $N$, then $\Xi$ can be taken to be $\fz_{\hh}(F)$ and $E + \Xi$
is the same as the Kostant section considered in Section \ref{subs:Transverse}.
We will abuse notation by referring to both the Kostant section defined
in Section \ref{subs:Transverse} and the section in Theorem \ref{theo:intKostant}
by $\kappa$.
\end{obs}

Theorem \ref{theo:intKostant} will be an important improvement from 
Theorem \ref{theo:Kostant}, since in the sequel we will need the Kostant
section to maintain certain integrality properties. In particular, it will
be helpful to apply Theorem \ref{theo:intKostant} over $\Z_p$, a feature
that would not be present if we only had Theorem \ref{theo:Kostant}.

\section{Constructing orbits}
\label{section:construct}

Given an element $b \in B(\Z)$ with discriminant weakly divisible by $m^2$
for a large squarefree number $m$ coprime to $N$, we will show how to 
construct a special $g \in G(\Z[1/m]) \setminus G(\Z)$
such that $g \kappa_b \in \frac{1}{N}V(\Z)$ in a way that ``remembers $m$''.

We start by defining the distinguished subspace $W_0 \subset V$ as 
\[
W_0 := \bigoplus_{\substack{a \in \Phi/\vartheta \\ \text{ht}(a) \leq 1}} V_a,
\]
where the notation is as in Section \ref{section:ResRoots}. We write an element
$v \in W_0(\Q)$ as $v = \sum_{\text{ht}(\alpha) = 1} v_{\alpha}X_{\alpha} + \sum_{\text{ht}(\beta) \leq 0} v_{\beta} X_{\beta}$,
where each $X_{\alpha},X_{\beta}$ generates each root space $V_{\alpha},V_{\beta}$
and $v_\alpha,v_\beta \in \Q$.
Then, we can define the \emph{$Q$-invariant} of $v \in W_0$ as $Q(v) = \left|\prod_{\text{ht}(\alpha) = 1} v_{\alpha}\right|$.
Now, define:
\[
W_M := \left\{v \in \frac{1}{N}V(\Z) \,\middle|\; v = g\kappa_b \text{ for a squarefree } m>M,\, (m,N) = 1, \, g \in G(\Z[1/m])\setminus G(\Z), \, b \in B(\Z), \Delta(b) \neq 0\right\}.
\]
The main result of the section is the following:
\begin{prop}
\label{prop:p-orbit}
Let $b \in B(\Z)$, and assume that $\Stab_{G(\Q)} \kappa_b = \{e\}$.
\begin{enumerate}
\item Let $m > M$ be a squarefree integer, coprime to $N$. If $m^2$ 
weakly divides $\Delta(b)$, then $W_M \cap \pi^{-1}(b)$ is non-empty.
\item If $v \in W_M \cap W_0$, then $Q(v) > M$.
\end{enumerate}
\end{prop}

The proof of Proposition \ref{prop:p-orbit} will rely on a reduction to
$\fsl_2$, inspired by the techniques in the proofs of \cite[Lemma 4.19]{LagaE6} 
and \cite[Proposition 5.4]{RTE8}, which we now explain.

Assume we have a connected reductive group $L$ over a field $k$, together
with an involution $\xi$. As in Section \ref{subs:Vinberg}, the Lie algebra $\mathfrak{l}$
decomposes as $\mathfrak{l} = \mathfrak{l}(0) \oplus \mathfrak{l}(1)$,
according to the $\pm 1$ eigenspaces of $d\xi$. We also write $L_0$
for the connected component of the fixed group $L^{\xi}$.
\begin{defi}
Let $k$ be algebraically closed. We say a vector $v \in \mathfrak{l}(1)$
is \emph{stable} if the $L_0$-orbit of $v$ is closed and its stabiliser $Z_{L_0}(v)$
is finite. We say $(L_0,\mathfrak{l}(1))$ is \emph{stable} if it contains stable
vectors. If $k$ is not necessarily algebraically closed, we say $(L_0,\mathfrak{l}(1))$ is \emph{stable}
if $(L_{0,k^s},\mathfrak{l}(1)_{k^s})$ is.
\end{defi}
By \cite[Proposition 1.9]{Thorne16}, the $\theta$ defined in Section \ref{subs:Vinberg}
is a stable involution, i.e. $(G,V)$ is stable.

We now prove the analogue of \cite[Lemma 2.3]{RTE8}: the proof is very
similar and is reproduced for convenience.
\begin{lemma}
\label{lemma:etale}
Let $S$ be a $\Z[1/N]$-scheme. Let $(L,\xi)$, $(L',\xi')$ be two pairs,
each consisting of a reductive group over $S$ whose geometric fibres are adjoint semisimple of
type $A_1$, together with a stable involution. Then for any $s \in S$ there exists
an étale morphism $S' \to S$ with image containing $s$ and an isomorphism
$L_{S'} \to L_S$ intertwining $\xi_{S'}$ and $\xi_{S'}'$.
\end{lemma}
\begin{proof}
We are working étale locally on $S$, so we can assume that $L = L'$ and
that they are both split reductive groups. Let $T$ denote the scheme
of elements $l \in L$ such that $\Ad(l) \circ \xi = \xi'$: by \cite[Proposition 2.1.2]{ConradRed},
$T$ is a closed subscheme of $L$ that is smooth over $S$. Since a surjective
smooth morphism has sections étale locally, it is sufficient to show
that $T \to S$ is surjective. Moreover, we can assume that $S = \Spec k$
for an algebraically closed field $k$, since the formation of $T$ commutes
with base change.

Let $A,A' \subset L$ be maximal tori on which $\xi,\xi'$ act as an automorphism
of order 2. By the conjugacy of maximal tori, we can assume that $A = A'$
and that $\xi,\xi'$ define the (unique) element of order 2 in the Weyl group.
Write $\xi = a \xi'$ for some $a \in A(k)$. Writing $a = b^2$ for some
$b \in A(k)$, we have $\xi = b \cdot b \cdot \xi' = b \cdot \xi' \cdot b^{-1}$.
The conclusion is that $\xi$ and $\xi'$ are $L(k)$-conjugate (in fact, $A(k)$-conjugate),
which completes the proof.
\end{proof}

The following lemma is the key technical part in our proof. We remark
the the first part was already implicitly proven in the proof of \cite[Theorem 7.17]{LagaThesis}.
\begin{lemma}
\label{lemma:sl2}
Let $p$ be a prime that does not divide $N$.
\begin{enumerate}
\item Let $b \in B(\Z_p)$ be an element with $\ord_{p} \Delta(b) = 1$,
where $\ord_p \colon \Q_p^{*} \to \Z$ is the usual normalized valuation.
Let $v \in V(\Z_p)$ with $\pi(v) = b$. Then, the reduction mod $p$ of $v$
in $V(\F_p)$ is regular.
\item Let $b \in B(\Z_p)$ be an element with discriminant weakly divisible
by $p^2$. Then,  there exists 
$g_{b,p} \in G(\Q_p) \setminus G(\Z_p)$ such that $g_{b,p} \cdot \kappa_b \in V(\Z_p)$.
\end{enumerate}
\end{lemma}
\begin{proof}
Let $v_{\F_p} = x_s+x_n$ be the Jordan decomposition of the reduction of
$v$ in $\F_p$. Then, we have a decomposition $\hh_{\F_p} = \hh_{0,\F_p} \oplus \hh_{1,\F_p}$,
where $\hh_{0,\F_p} = \fz_{\hh}(x_s)$ and $\hh_{1,\F_p} = \text{image}(\Ad(x_s))$.
By Hensel's lemma, this decomposition lifts to $\hh_{\Z_p} = \hh_{0,\Z_p} \oplus \hh_{1,\Z_p}$,
with $\ad(v)$ acting topologically nilpotently in $\hh_{0,\Z_p}$ and
invertibly in $\hh_{1,\Z_p}$. As explained in the proof of \cite[Lemma 4.19]{LagaE6},
there is a unique closed subgroup $L \subset H_{\Z_p}$
which is smooth over $\Z_p$ with connected fibres and with Lie algebra
$\hh_{0,\Z_p}$.

For the first part of the lemma, we are free to replace $\Z_p$ for a
complete discrete valuation ring $R$ with uniformiser $p$, containing $\Z_p$
and with algebraically closed residue field $k$. In this case, the spreading out
properties in \cite[\S 7.2]{LagaThesis} guarantee that the derived group of $L$ is of type
$A_1$. Since the restriction of $\theta$ restricts to a stable involution
in $L$ by \cite[Lemma 2.5]{ThorneThesis}, Lemma \ref{lemma:etale} guarantees that there exists an 
isomorphism $\hh_{0,R}^{der} \cong \fsl_{2,R}$ intertwining the action 
of $\theta$ on $\hh_{0,R}^{der}$ with the action of $\xi = \Ad(\diag(1,-1))$ on $\fsl_{2,R}$.
To show that $v_k$ is regular is equivalent to showing that the
nilpotent part $x_n$ is regular in $\hh_{0,k}^{der}$. The elements $v_k$
and $x_n$ have the same projection in $\hh_{0,k}^{der}$, and given that
$v\in \hh_{0,R}^{der,d\theta = -1}$, its image in $\fsl_{2,R}$ is of the form
\[
\begin{pmatrix}
0 & a \\ b & 0
\end{pmatrix}.
\]
We claim that $\ord_R(ab) = 1$. This can be seen from an argument similar to
the end of \cite[Lemma 7.15]{LagaThesis}, i.e. using \cite[Lemma 2.3]{LagaThesis}\footnote{The cited result is only stated in \cite{LagaThesis} for fields of characteristic zero, but it is still valid in our situation for $k$: the only results that are invoked in that proof are those in \cite[\S 3]{Steinberg}, which hold as long as $\Char k$ is not a torsion prime for $H$, which we may assume.}, 
it follows that the discriminant of $v$ in $\hh$ coincides with the discriminant
of its image in $\fsl_2$ up to a unit in $R$, as wanted.
In particular, exactly one of $a,b$ is non-zero when reduced to $k$, and
hence $x_n$ is regular in $\hh_{0,k}^{der}$, as wanted.

For the second part, we return to the case $R = \Z_p$. If $b \in B(\Z_p)$
has discriminant weakly divisible by $p^2$, there exists $b' \in B(\Z_p)$
such that $\ord_p\Delta(b+pb') = 1$. Since the Kostant section $\kappa$
is algebraic, we know that $\kappa_b - \kappa_{b+pb'} \in pV(\Z_p)$.
By the first part of the lemma, we know that
$\kappa_{b+pb'}$ is regular mod $p$, and so $\kappa_b$ is also regular mod $p$. In particular,
writing $\kappa_{b,\F_p} = x_s + x_n$ as before,
this means that the nilpotent part $x_n$ is a regular nilpotent in $\hh_{0,\F_p}^{der}$.
We now claim that:
\begin{enumerate}[(i)]
\item We have an isomorphism $\hh_{0,\Z_p}^{der} \cong \fsl_{2,\Z_p}$;
\item The isomorphism intertwines the actions of $\theta$ and the previously
defined $\xi$;
\item Over $\F_p$, the isomorphism sends the regular nilpotent $x_n$ to
the matrix
\[
e = 
\begin{pmatrix}
0 & 1 \\ 0 & 0
\end{pmatrix}
\]
of $\fsl_{2,\F_p}$.
\end{enumerate}
We note that this does not follow immediately from 
Lemma \ref{lemma:etale}, as the isomorphism \textit{a priori} does not
need to be defined over $\Z_p$. 

We prove our claim as follows:
Consider the $\Z_p$-scheme $X = \Isom((L/Z(L),\theta),(\PGL_2,\xi))$,
consisting of isomorphisms between $L/Z(L)$ and $\PGL_2$ that
intertwine the $\theta$ and $\xi$-actions. Using Lemma \ref{lemma:etale},
we see that étale-locally, $X$ is isomorphic to $\Aut(\PGL_2,\xi)$; in
particular, it is a smooth scheme over $\Z_p$. By Hensel's lemma \cite[Théorème 18.5.17]{EGA4},
to show that $X$ has a $\Z_p$-point it is sufficient to show that it has
an $\F_p$-point.

We now consider the $\F_p$-scheme $Y = \Isom((L/Z(L)_{\F_p},\theta,x_n),(\PGL_2,\xi,e))$
of isomorphisms preserving the $\theta$ and $\xi$-actions which send $x_n$ to $e$:
it is a subscheme of $X_{\F_p}$. If $Y(\F_p)$ is non-empty, then by
Hensel's lemma it can be lifted to an isomorphism of $X(\Z_p)$ satisfying
all three points of the claim. Therefore, the claim will follow from seeing
that $Y(\F_p) \neq \emptyset$.

Again by Lemma \ref{lemma:etale},
$Y$ is étale locally of the form $\Aut(\PGL_2,\xi,e)$, since $\PGL_2^{\xi}$
acts transitively on the regular nilpotents of $\fsl_2^{d\xi = -1}$ for
any field of characteristic $p > N$. In particular, we see that $Y$ is an 
$\Aut(\PGL_2,\xi,e)$-torsor. In this situation, to see that $Y(\F_p)$ is 
non-empty it will suffice to see 
that $\Aut(\PGL_2,\xi,e) = \Spec \F_p$. This follows from the elementary
computation of the stabiliser of $e$ under $\PGL_2^\xi$, which can be seen
to be trivial over any field.

In conclusion, $Y(\F_p)$ is non-empty, meaning that there is an isomorphism
$\hh_{0,\Z_p}^{der} \cong \fsl_{2,\Z_p}$ respecting $\theta$ and $\xi$,
and we can make it so that the projection of $\kappa_b$ in $\fsl_{2,\Z_p}$ is an
element of the form
\[
\begin{pmatrix}
0 & a \\ bp^2 & 0
\end{pmatrix},
\]
with $a,b \in \Z_p$ and $a \in 1+p\Z_p$. Moreover, there exists a morphism
$\varphi \colon \SL_2 \to L^{der}_{\Q_p}$ inducing the given isomorphism
$\hh_{0,\Q_p}^{der} \cong \fsl_{2,\Q_p}$, since $\SL_2$ is simply connected. 
The morphism $\varphi$ necessarily respects the grading, and induces a
map $\SL_2(\Q_p) \to L^{der}(\Q_p)$ on the $\Q_p$-points. Consider the 
matrix $g_{b,p}=\varphi(\diag(p,p^{-1}))$: it satisfies the conditions 
of the lemma, and so we are done.
\end{proof}
\begin{obs}
A natural follow-up question to Lemma \ref{lemma:sl2} is to ask how many
$g_{b,p} \in G(\Q_p) \setminus G(\Z_p)$ are there (up to a $G(\Z_p)$-action)
such that $g_{b,p}\cdot k_b \in V(\Z_p)$. The proof of the lemma implies
that if $p^k \mid \Delta(b)$, then the projection of $\kappa_b$ in
$\fsl_{2,\Z_p}$ is of the form $\begin{pmatrix}
0 & a \\ bp^k & 0
\end{pmatrix}$, so we can conjugate by $\diag(p,p^{-1})$ a total of
$\lfloor\frac{k}{2}\rfloor$ times. It is natural to expect that all the
possible choices of $g_{b,p}$ arise in this fashion; however, we do not
know if that is true.
\end{obs}
\begin{obs}
It would be very convenient if in the proof of Lemma \ref{lemma:sl2} we
could obtain a $g \in \SL_2(\Q_p)$ such that
\[
g \begin{pmatrix} 0 & a \\ bp^2 & 0 \end{pmatrix} g^{-1} = \begin{pmatrix} 0 & ap \\ bp & 0 \end{pmatrix},
\]
in order to transform the ``mod $p^2$'' divisibility into ``mod $p$''
divisibility, but unfortunately that doesn't appear to be possible in
general. If that were the case, the element $v' \in V(\Z_p)$ corresponding to the
matrix $\begin{pmatrix} 0 & ap \\ bp & 0 \end{pmatrix}$ would not be
regular modulo $p$, and in this situation we would be able to count such
orbits using \cite{BEkedahl} (without needing geometry-of-numbers!).
We note that this strategy is used in \cite[Proof of Theorem 6.10]{RTE8}, which
works in their case because they are working over a $\Z/3\Z$-grading instead
of a $\Z/2\Z$-grading.
\end{obs}

\begin{proof}[Proof of Proposition \ref{prop:p-orbit}]
We start by proving the first item. Since $G$ has class number 1 by Proposition \ref{prop:cl1}, the 
natural map $G(\Z) \backslash G(\Z[1/m]) \to \prod_{p \mid m} G(\Z_p)\backslash G(\Q_p)$
is a bijection. In Lemma \ref{lemma:sl2}, for each prime $p \mid m$, we
constructed an element $g_{b,p} \in G(\Z_p) \backslash G(\Q_p)$, so all
these elements together correspond to some element $g_b \in G(\Z[1/m]) \setminus G(\Z)$.
By construction, $g_b \cdot \kappa_b$ belongs to $(\cap_{p \mid m} V(\Z_p)) \cap V(\Z[1/m]) = V(\Z)$.

We now prove the second item. Specifically, if $v \in W_M \cap W_0$ is
given by $g\kappa_b$ for some $g \in G(\Z[1/m])\setminus G(\Z)$, we will 
prove that $m \mid Q(v)$. It suffices to consider each prime $p \mid m$ separately,
so assume that $g \in G(\Z[1/p]) \setminus G(\Z)$. Since the group $H$ is adjoint, there exists
a $t \in T(\Q)$ that makes all the height-one coefficients of $t\kappa_b$ be equal to one,
and in this case we see that $t \in T^{\theta}(\Q)$. By Theorem \ref{theo:intKostant},
there exists a unique $\gamma \in N^{-}(\Q)$ such that $\gamma t \kappa_b = v$; 
by taking $\theta$-invariants in the isomorphisms of Theorem \ref{theo:intKostant}. 
we see that $\gamma \in N^{-,\theta}(\Q)$. Since the stabiliser is trivial,
we see that $g = \gamma t$, or in other words that $g \in P^{-,\theta}(\Z[1/p]) \setminus P^{-,\theta}(\Z)$.

Assume that $Q(v)$ is invertible in $\Z_p$, so that all the height-one coefficients
of $v$ are invertible. Then, there exists a $t' \in T(\Z_p)$ making all
the height-one coefficients of $t'v$ be equal to one, and by Theorem \ref{theo:intKostant}, there exists
at most one element $\gamma'$ in $N^{-}(\Z_p)$ such that $\gamma' t'\kappa_b = v$.
Consequently, $g \in P^{-,\theta}(\Z_p) \cap P^{-,\theta}(\Z[1/p]) = P^{-,\theta}(\Z)$, a contradiction.
In summary, we have that $p \mid Q(v)$ for all primes $p \mid m$, as wanted.
\end{proof}

\begin{ex}
Our construction is inspired by the construction in
\cite[Sections 2.2 and 3.2]{BSWsquarefree} for the case $A_n$. In that
case, $C \to B$ corresponds to the family of hyperelliptic curves $y^2 = f(x)$,
where $f(x)$ has degree $n+1$ (there is a slight difference between this
paper and \cite{BSWsquarefree}, in that we consider $f(x)$ without an $x^n$
term while they consider polynomials with a possibly non-zero linear term;
we ignore this difference for now). The main goal of \cite[Sections 2.2 and 3.2]{BSWsquarefree}
is to construct an embedding
\[
\sigma_m \colon \cW_2^{(m)} \to \frac{1}{4}W_0(\Z) \subset \frac{1}{4}V(\Z) ,
\]
where $\sigma_m(f)$ has characteristic polynomial $f$ and $Q(\sigma_m(f)) = m$.\footnote{In \cite{BSWsquarefree}, the space that we denote as $W_0$ is denoted there by $W_{00}$.}
By taking the usual pinning in $\SL_{n+1}$, we see that $V$ corresponds to
the space of matrices in $\fsl_{n+1}$ which are symmetric across the
antidiagonal, $W_0$ corresponds to the subspace of $V$ where the
entries above the superdiagonal are zero, and the height-one entries are
precisely those in the superdiagonal (in \cite{BSWsquarefree}, everything
is ``reflected vertically'', so for instance $V$ is the space of symmetric
matrices across the diagonal; this makes no difference in the results).
An explicit section of $B$ can
be taken to lie in $\frac{1}{4}W_0(\Z)$: namely, if $n$ is odd, the matrix
\[
B(b_1,\dots,b_{n+1}) = \begin{pmatrix}
0&1&&&&&&&\\
&0&\ddots&&&&&&\\
&&&1&&&&&\\
&&&0&1&&&&\\
&&&\frac{-b_2}{2}&-b_1&1&&&\\
&&\iddots&-b_3&\frac{-b_2}{2}&0&1&&\\
&\frac{-b_{n-2}}{2}&\iddots&\iddots&&&&\ddots&\\
\frac{-b_{n}}{2}&-b_{n-1}&\frac{-b_{n-2}}{2}&&&&&0&1\\
-b_{n+1}&\frac{-b_{n}}{2}&&&&&&&0
\end{pmatrix}
\]
can be seen to have characteristic polynomial $f(x) = x^{n+1} + b_1x^n +\dots+b_nx +b_{n+1}$;
if $n$ is even, a similar matrix can be given.
The main observation in this case is that if $m^2$ weakly divides $\Delta(f)$,
then there exists an $l\in \Z$ such that $f(x+l) = x^{n+1} + p_1x^n +\dots+mp_nx +m^2p_{n+1}$
(cf. \cite[Proposition 2.2]{BSWsquarefree}). Then, if $D = \diag(m,1,\dots,1,m^{-1})$,
we observe that the matrix
\[ 
D(B(p_1,\dots,p_{n-1},mp_n,m^2p_{n+1})+lI_{n+1})D^{-1}
\]
is integral, has characteristic polynomial $f(x)$ and the entries in the 
superdiagonal are $(m,1,\dots,1,m)$. Thus, this matrix has $Q$-invariant $m$, as desired.
\end{ex}
\begin{obs}
Our $Q$-invariant is slightly different to the $Q$-invariant defined in
\cite{BSWsquarefree}, which is defined in a slightly more general subspace of $V$. When
restricting to $W_0(\Q)$, their $Q$-invariant turns out to be a product
of powers of the elements of the superdiagonal, whereas in our case we
simply take the product of these elements. This difference does not affect 
the proof of Theorem \ref{theorem:tail}, and we can also see that for both 
definitions the $Q$-invariant in the previous example is $m$.
\end{obs}

\section{Reduction theory}

In light of the results in Section \ref{section:construct}, 
to bound families of curves with non-squarefree discriminant it
is sufficient to estimate the size of the $G(\Z)$-invariant set $W_M$.
Before we are able to obtain such an estimate, we will need to obtain
a precise count of the number of reducible $G(\Z)$-orbits in $V(\Z)$.
To do so, we will first need some results 
about reduction theory: most importantly, we will construct a box-shaped
domain for the action of $G(\Z)$ on $G(\R)$, in the style of \cite[\S 2.2]{SSSVcusp}.

\subsection{Heights}

Recall that $B = \Spec \Z[p_{d_1},\dots,p_{d_r}]$. For any $b \in B(\R)$,
we define the \emph{height} of $b$ to be
\[
\text{ht}(b) = \sup_{i=1,\dots,r} |p_{d_i}(b)|^{1/d_i}.
\]
Similarly, for every $v \in V(\R)$ we define $\text{ht}(v) := \text{ht}(\pi(v))$.
We record the following fact from \cite[Lemma 8.1]{LagaThesis}, which in
particular means that the number of elements of $B(\Z)_{<X} := \{b \in B(\Z) \mid \Ht(b) < X\}$ 
is of order $X^{\dim V}$:
\begin{lemma}
We have $d_1+\dots+d_r = \dim_{\Q}V$.
\end{lemma}

\subsection{Measures on $G$}
\label{subs:measures}

Let $\Phi_G = \Phi(G,T^{\theta})$ be the set of roots of $G$. The Borel
subgroup $P^{\theta}$ of $G$ determines a root basis $S_G$ and a set of
positive/negative roots $\Phi_G^{\pm}$, compatible with the choice of positive
roots in $H$ determined by the pinning of Section \ref{subs:Vinberg}. 
Let $\ol{N}$ be the unipotent 
radical of the negative Borel subgroup $P^{-,\theta}$. Then, there exists
a maximal compact subgroup $K \subset G(\R)$ such that
\[
\ol{N}(\R) \times T^{\theta}(\R)^{\circ} \times K \to G(\R)
\]
given by $(n,t,k) \mapsto ntk$ is a diffeomorphism; see \cite[Chapter 3, \S 1]{LangSL2}.
We can choose $K$ to be ``compatible'' with $T$; that is, we can choose
a Cartan involution $\tau$ such that the fixed points of $G$ with respect
to $\tau$ is exactly $K$, and satisfying that $\tau|_T$ is just the inversion map.
The following result is a well-known property of the Iwasawa decomposition:
\begin{lemma}
\label{lemma:IwasawaInt}
Let $dn,dt,dk$ be Haar measures on $\overline{N}(\R),T^{\theta}(\R)^{\circ},K$,
respectively. Then, the assignment
\[
f \mapsto \int_{n \in \overline{N}(\R)}\int_{t \in T^{\theta}(\R)^{\circ}}\int_{k \in K} f(ntk) \delta(t)^{-1} dn\, dt\, dk
\]
defines a Haar measure on $G(\R)$. Here, $\delta(t) = \prod_{\beta \in \Phi_G^{-}} \beta(t) = \det \Ad(t)|_{\Lie \overline{N}(\R)}$.
\end{lemma}

We get the measure on $T^{\theta}(\R)^{\circ}$
by pulling it back from the isomorphism $\prod_{\beta \in S_G} \beta \colon T^{\theta}(\R)^{\circ} \to \R_{>0}^{\# S_G}$,
where $\R_{>0}$ is given the standard Haar measure $d^{\times}\lambda = d\lambda/\lambda$.
We will choose the normalizations for $dn$ and $dk$ in Section \ref{subs:averaging}
in a way that will be convenient for us.

\subsection{Fundamental domains}
\label{subs:FDomains}

In this section, we construct a fundamental domain for the action of
$G(\Z)$ on $G(\R)$. In view of \cite{SSSVcusp}, it will be useful to construct
a ``box-shaped'' fundamental domain $\cF$, which we will now define.
For any $c > 0$, define $T_c = \{t \in T^{\theta}(\R)^{\circ} \mid \forall \alpha \in S_G,\;\alpha(t) \geq c\}$.
We define a \emph{Siegel set} to be a set of the form $\cS = \omega \cdot T_c \cdot K$,
where $\omega \subset \ol{N}(\R)$ is a compact subset, $c$ is a positive
real constant and $K$ is the maximal compact subset fixed in Section \ref{subs:measures}.
Then, we say that a fundamental domain $\cF$ for the action of $G(\Z)$
on $G(\R)$ is \emph{box-shaped at infinity} if there exist two Siegel sets
$\cS_1 \subset \cF \subset \cS_2$ satisfying that:
\begin{enumerate}
\item There exists an open subset $\cU_1 \subset \cS_1$ of full measure such that every 
$G(\Z)$-orbit in $G(\R)$ intersects $\cU_1$ at most once.
\item Every $G(\Z)$-orbit in $G(\R)$ intersects $\cS_2$ at least once.
\item For sufficiently large $c$, we have $\cS_1 \cap \ol{N}T_cK = \cS_2 \cap \ol{N}T_cK$.
\end{enumerate}
To construct $\cF$, we will see that it is sufficient to construct $\cS_1$
and $\cS_2$. More precisely, we have as in \cite[Lemma 7]{SSSVcusp}:
\begin{lemma}
\label{lemma:box}
Let $\Lambda$ be a discrete subgroup of a Lie group $G$ and denote by
$B(G)$ the Borel $\sigma$-algebra of $G$. Assume there exist sets $\cS_1,\cS_2$
in $B(G)$ such that the natural maps $\cS_1 \to G/\Lambda$ and $\cS_2 \to G/\Lambda$
are injective and surjective, respectively. Then, there exists a set $\cF$
in $B(G)$ which is a fundamental domain for the action of $\Lambda$ on
$G$ satisfying $\cS_1 \subset \cF \subset \cS_2$.
\end{lemma}
We will construct $\cS_1$ and $\cS_2$ in the following subsections.
\begin{obs}
The constructed $\cS_1$ and $\cS_2$ will not strictly be Siegel sets of
the form $\omega T_c K$, but rather of the form $\omega T_c K'$ for some
subset $K'$ of $K$. We will call them Siegel sets regardless.
\end{obs}

\subsubsection{Constructing $\cS_1$}
\label{subs:corners}

To obtain the domain $\cS_1$, we will use general properties of the
Borel-Serre compactification following \cite{BScorners}.
The construction below holds for a general connected semisimple
algebraic group $G$ over $\Q$, unless otherwise specified (note that our group $G$ is always semisimple
by \cite[Proposition 3.7]{LagaThesis}).

Consider the symmetric space $X = G(\R)/K$, where $K$ is a maximal compact
subgroup of $G(\R)$. For each parabolic $\Q$-subgroup $P$ of $G$, let
$S_P := (R_dP/(R_uP \cdot R_dG))$, where $R_u$ denotes the unipotent radical
and $R_d$ denotes the $\Q$-split part. Then, $S_P$ is a $\Q$-split torus,
and we let $A_P := S_P(\R)^{\circ}$. There is a natural action of $A_P$ on $X$
called the geodesic action (see \cite[(3.2)]{BScorners}). Set $e(P) = A_P \backslash X$,
and consider
\[
\overline{X} = \coprod_{P \text{ parabolic}} e(P),
\]
which by \cite[(7.1)]{BScorners} naturally has a structure of a manifold
with corners. The topology of $\overline{X}$ is studied in \cite[\S 5, \S 6]{BScorners};
in particular, it is shown that for any parabolic group $P$, the subset
$X(P) = \coprod_{Q \supset P} e(Q)$ is an open subset of $\overline{X}$.
Taking $P = G$, we see that $e(G) = X$ is an open submanifold of $\overline{X}$.

Assume for simplicity that $G$ is split over $\Q$ with split maximal torus $T$.
Let $P = \ol{N}T$ be a Borel subgroup of $G$. For $x \in X$ and a real constant $c > 0$, we can consider the set
\[
U_{x,P,c} = \ol{N}(\R) (T_c \cdot x).
\]
Its closure $\overline{U_{x,P,c}}$ in $\overline{X}$ is a neighbourhood
of the closure of $e(P)$ in $\overline{X}$. Then, we have the following
result (see \cite[Proposition 10.3]{BScorners}):
\begin{prop}
\label{prop:corners}
There exists $c > 0$ satisfying that for any $g_1, g_2 \in \overline{U_{x,P,c}}$, if there
exists $\gamma \in G(\Z)$ such that $g_1 = \gamma g_2$, then $\gamma \in P(\Z)$.
\end{prop}

To obtain a suitable Siegel set $\cS_1$, we need to carefully choose a
compact subset $\omega \subset \ol{N}(\R)$. Let $(\alpha_1,\dots,\alpha_k)$
be an ordering of the positive roots of $G$ satisfying that $\Ht(\alpha_i) \leq \Ht(\alpha_{i+1})$
for all $1 \leq i \leq k-1$. For each root $\alpha_i$ we consider
the isomorphism $u_{\alpha_i} \colon \G_a \to U_{\alpha_i}$, where $U_{\alpha_i} \subset \ol{N}$.
By \cite[Theorem 5.1.13]{ConradRed}, there is an isomorphism of 
varieties over $\Z$:
\[
\prod_{i=1}^k U_{\alpha_i} \to \ol{N}
\]
which is just the multiplication map. In other words, we can express any
element of $\ol{N}(\R)$ as $u_{\alpha_1}(x_1)\cdots u_{\alpha_k}(x_k)$ for some $x_1,\dots,x_k \in \R$.
Moreover, a set of $x_1,\dots,x_k$ will correspond to an element of $\ol{N}(\Z)$
if and only if $x_1,\dots,x_k \in \Z$. We now recall the following
result (see e.g. \cite[Proposition 5.1.14]{ConradRed}):
\begin{lemma}
\label{lemma:commute}
Let $x,y \in \R$, and let $\alpha,\beta$ be positive roots. Then,
\[
u_{\alpha}(x)u_{\beta}(y)u_{\alpha}(-x)u_{\beta}(-y) = \prod_{i,j > 0} u_{i\alpha + j\beta} (c_{i,\alpha,j,\beta}x^iy^j).
\]
Here $c_{i,\alpha,j,\beta}$ is a constant, and the product is taken over all $i,j > 0$ such that $i\alpha+j\beta$
is a positive root.
\end{lemma}

Consider the set $\ol{\omega} = \{u_{\alpha_1}(x_1)\cdots u_{\alpha_k}(x_k) \in \ol{N}(\R) \mid x_i \in [-1/2,1/2] \, \forall i\} \subset \ol{N}(\R)$.

\begin{prop}
We have that
\begin{enumerate}
\item $\ol{N}(\Z) \ol{\omega} = \ol{N}(\R)$.
\item Except for a set of zero measure, no two distinct elements of $\ol{\omega}$ are
$\ol{N}(\Z)$-translates of each other.
\end{enumerate}
\end{prop}
\begin{proof}
For the first point, let $y_1,\dots,y_k \in \R$. We will show that there exist $n_1,\dots,n_k \in \Z$
and $x_1,\dots,x_k \in \R$ such that
\begin{equation}
\label{eq:ualpha}
u_{\alpha_1}(n_1)\cdots u_{\alpha_k}(n_k) u_{\alpha_1}(x_1)\cdots u_{\alpha_k}(x_k) = u_{\alpha_1}(y_1)\cdots u_{\alpha_k}(y_k).
\end{equation}
Using the commutator relations of Lemma \ref{lemma:commute}, we can reorder
the terms in the left hand side to get equations of the form
\begin{equation}
\label{eq:roots}
y_m = n_m + x_m + p_m(n_1,\dots,n_k,x_1,\dots,x_k),
\end{equation}
where $p_m$ are polynomials. By examining the commutator relations,
we see that $p_m$ only depends on the variables corresponding to lower
height coefficients. In particular, if $\alpha_m$ is a height-one root,
we can choose $n_m \in \Z$ and $x_m \in [-1/2,1/2]$ such that $y_m = n_m + x_m$.
We can then find coefficients $n_m,x_m$ for the larger height roots inductively
using \eqref{eq:roots}.

For the second point, choose two elements of $\ol{\omega}$ with coefficients
$x_1,\dots,x_k$ and $y_1,\dots,y_k$ lying in $(-1/2,1/2)$. Assume there
exist $n_1,\dots,n_k \in \Z$ satisfying \eqref{eq:ualpha}. By induction,
we will show that $n_i = 0$ for all $i$. This is clear for the height-one
coefficients, since $n_i + x_i = y_i$. Assume by induction that all the
coefficients $n_i$ are zero up to some height $h$. We note that by Lemma \ref{lemma:commute}
all terms in the polynomial $p_m(n_1,\dots,n_k,x_1,\dots,x_k)$ are multiple
of at least one $n_i$ of lower height.
Hence, by induction we get that $p_m(n_1,\dots,n_k,x_1,\dots,x_k) = 0$
and thus that $n_m = 0$, as wanted.
\end{proof}

Assume from now on that our group $G$ is one of the groups constructed
in Section \ref{subs:Vinberg}.
We note that $T^{\theta}(\Z)$ acts by conjugation on $\ol{\omega}$: recall that $\Ad(t) \cdot u_{\alpha}(x) = u_{\alpha}(\alpha(t)\cdot x)$,
and for any $t \in T^{\theta}(\Z)$ we have that $\alpha(t) = \pm 1$.
Alternatively, we can say that there is a mapping
$T^{\theta}(\Z) \to \{\pm 1\}^{\# S_G}$ given by $t \mapsto (\alpha(t))_{\alpha \in S_G}$;
however, it needs not be surjective: denote by $A = \{a_1,\dots,a_l\}$
a set of representatives of the cokernel of this map. For any element
$a_i \in A$, write it as $a_i = (a_{i,1},\dots,a_{i,k})$, where $a_{i,j} = \pm 1$
correspondingly. Consider the set $\omega_i$ inside $\ol{\omega}$ consisting
of those elements $u = u_{\alpha_1}(x_1)\cdots u_{\alpha_k}(x_k)$ such that
for all height-one coefficients $\alpha_j$, we have $x_j \in a_{i,j} \cdot [0,1/2]$.
Finally, define $\omega$ to be the union of the sets $\omega_i$.
Then, each element in $\ol{\omega}$ is conjugate to a unique element in
$\omega$.

Additionally, we note that $T^{\theta}(\Z) \subset K$, since $T^{\theta}(\Z)$ is fixed by 
the Cartan involution $\tau$ chosen in Section \ref{subs:measures},
and $\tau|_{T^{\theta}}$ is just the inverse map.

Take $\cS_1 = \omega T_c \ol{K}$, where $c > 0$ satisfies the conclusions of
Proposition \ref{prop:corners}, and $\ol{K}$ is a fundamental set for 
the action of $Z(G)(\Z)$ on $K$. Let $g_1 = n_1t_1k_1$ and $g_2 = n_2t_2k_2$
be two elements of $\cS_1$, and moreover we assume that $n_1$ and $n_2$
lie in the interior of $\omega$ (this interior is a set of full measure).
Assume that $g_1$ and $g_2$ are equivalent under the $G(\Z)$-action. By
Proposition \ref{prop:corners}, it follows that $g_1$ and $g_2$ have to
be $P^{\theta}(\Z)$-conjugate, say by an element $p_0 = n_0 t_0$ for $n_0 \in \ol{N}(\Z)$
and $t_0 \in T^{\theta}(\Z)$. Then, we can write
\[
n_0 (t_0 n_1 t_0^{-1}) t_1 (t_0k_1) = n_2t_2k_2.
\]
By uniqueness in the Iwasawa decomposition, we have that $n_0 (t_0 n_1 t_0^{-1}) = n_2$,
$t_1 = t_2$ and $t_0 k_1 = k_2$. If we look at the first equation in terms
of height-one roots $\alpha_i$, we get equalities of the form 
$u_{\alpha_i}(x_0)u_{\alpha_i}(\alpha_i(t_0)x_1) = u_{\alpha_i}(x_2)$, where
$x_0 \in \Z$ and $x_1,x_2 \in [-1/2,1/2]$ (or a subinterval if appropriate).
This can only happen if $x_0 = 0$ for all coefficients, meaning that $n_0 = 1$,
and also by construction of $\omega$ it must also happen that $\alpha_i(t_0) = 1$
for all $i$, or in other words that $t_0 \in Z(G)(\Z)$. Then, the last equation
$t_0 k_1 = k_2$ can only happen if $t_0 = 1$ by construction. Therefore, $g_1 = g_2$
as wanted.

\subsubsection{Constructing $\cS_2$}
We can construct $\cS_2$ compatibly with $\cS_1$ thanks to the following proposition:

\begin{prop}
There exists a real constant $c > 0$ such that $G(\R) = G(\Z) \omega T_c \ol{K}$,
where $\omega$ and $\ol{K}$ are as in Section \ref{subs:corners}.
\end{prop}
\begin{proof}
We can show that $G(\R) = G(\Z) \omega' T_c K$ for some compact subset
$\omega' \subset N(\R)$ and some $c > 0$ using \cite[Theorem 4.15]{PlatonovRapinchuk}, the first statement is reduced to showing
that $G(\Q) = P^{\theta}(\Q)G(\Z)$, which follows from \cite[\S 6, Lemma 1(b)]{Borel66}.

It is clear that $K$ can be substituted by $\ol{K}$, since we can multiply
by an appropriate element of $Z(G)(\Z)$ in $G(\Z)$. Now, let $g = g_0 n t k$
be an element of $G(\R) = G(\Z) \omega' T_c \ol{K}$: we will show that
$g \in G(\Z) \omega T_c \ol{K}$. We know that there exists $n_0 \in N(\Z)$
and $t_0 \in T^{\theta}(\Z)$ such that $t_0 n_0 n t_0^{-1} \in \omega$. Let $z \in Z(G)(\Z)$
be such that $zt_0k \in \ol{K}$. Then, $g = (z^{-1}g_0n_0^{-1}t_0^{-1})(t_0n_0nt_0^{-1})t(zt_0k) \in G(\Z) \omega T_c \ol{K}$,
as wanted.
\end{proof}

We fix $\cS_2 = \omega T_c \ol{K}$, for some $c > 0$ satisfying the 
above proposition. It is clear
then that $\cS_1$ and $\cS_2$ satisfy the required properties, and hence
that by Lemma \ref{lemma:box} we obtain a box-shaped fundamental domain $\cF$ for
the action of $G(\Z)$ over $G(\R)$.

\section{Counting reducible orbits}
\label{section:Counting}

In light of the results of Section \ref{section:construct}, to estimate
the elements of $B(\Z)$ having discriminant divisible by the square of a
large prime, it suffices to count certain special reducible $G(\Z)$-orbits
in $V(\Z)$. In this section, we develop much of what we will need in this
regard, following Bhargava's geometry-of-numbers techinques, and in particular
using the ideas in \cite{SSSVcusp}.

\subsection{Averaging}
\label{subs:averaging}

Let $S \subset V(\Z)$ be a $G(\Z)$-invariant subset. Define
\[
N(S,X) = \sum_{\substack{v \in G(\Z) \backslash S \\ \Ht(v) < X}} \frac{1}{\# \Stab_G(v)(\Z)}.
\]
We will prove the following:
\begin{theorem}
\label{theo:reducible}
There exist real positive constants $C,\delta$ such that
\[
N(V(\Z)^{red},X) = C X^{\dim V} + O(X^{\dim V - \delta}).
\]
\end{theorem}
By analogous arguments to \cite[\S 2.9]{ThorneE6}, there exist open subsets $L_1,\dots,L_k$
covering $\{b \in B(\R) \mid \Ht(b) = 1,\, \Delta(b) \neq 0\}$ such that
for a fixed $i$, the quantity $r_i = \# \Stab_{G(\R)}(v)$ remains constant for 
any choice of $v \in \pi^{-1}(L_i)$. We will denote $\Lambda = \R_{> 0}$
and $V_i := V(\Z)^{red} \cap G(\R) \kappa(\Lambda L_i)$.
Fix a compact left and right $K$-invariant set $G_0 \subset G(\R)$ which is the
closure of a non-empty open set, for which we assume that $G_0 = G_0^{-1}$.
An averaging argument just as in \cite[\S 2.3]{BSquartics} yields
\begin{equation}
\label{eq:averaging}
N(V_i,X) = \frac{1}{r_i \vol(G_0)} \int_{g \in \cF} \# \{v \in V(\Z)^{red} \cap (g G_0 \kappa(\Lambda L_i))_{< X} \} dg. 
\end{equation}
To obtain the estimate for $N(V(\Z)^{red},X)$, it will suffice to obtain
the appropriate estimates for $N(V_i,X)$. For any subset $S$ inside
$V(\Z)^{red} \cap G(\R) \kappa(\Lambda L_i)$, we can use the expression
\eqref{eq:averaging} to \emph{define} $N(S,X)$ as
\[
N(S,X) = \frac{1}{r_i \vol(G_0)} \int_{g \in \cF} \# \{v \in S \cap (g G_0 \kappa(\Lambda L_i))_{< X} \} dg. 
\]

For the argument, it will be crucial to use Davenport's lemma (see \cite{Davenport}), 
as stated in \cite[Proposition 2.6]{BSquartics}. We record it here for
convenience.
\begin{prop}
\label{prop:Davenport}
Let $\cR$ be a bounded, semialgebraic multiset in $\R^n$ having maximum
multiplicity $m$ and that is defined by at most $k$ polynomial inequalities,
each having degree at most $l$. Then,
\[
\# (\cR \cap \Z^n) = \vol(\cR) + O(\max(\{\vol(\overline{\cR}),1\})),
\]
where $\vol(\overline{\cR})$ denotes the greatest $d$-dimensional volume of
any projection of $\cR$ onto a coordinate subspace obtained by equating
$n-d$ coordinates to zero, and where $d$ takes any value between $1$ and
$n-1$. The implied constant in the second summand depends only on $n,m,k$
and $l$.
\end{prop}

\subsection{Applying the Selberg sieve}
\label{subs:Selberg}
Another important step in our argument will be the use of the Selberg sieve.
Notably, in the statement of Theorem \ref{theo:reducible} we require a
power saving estimate in the error term, which we will obtain by applying
the Selberg sieve as in \cite{STSelberg}. In this section, we describe
exactly how the Selberg sieve is used, and which hypothesis are needed.

The general situation is the following: suppose we have a finite sequence of
non-negative numbers $\cA = (a_n)_n$, and let $P$ be a finite
product of distinct primes. For all $d |P$, assume the following holds:
\begin{equation}
\label{eq:SelbergHyp}
\sum_{n \equiv 0 \bmod{d}} a_n = g(d)X + r_d,
\end{equation}
where $X > 0$ and $g(d)$ is a multiplicative function satisfying $0 < g(p) < 1$
for all primes $p|P$. Define the multiplicative function $h$ by $h(p) = \frac{g(p)}{1-g(p)}$
at primes $p$. For some choice of $D_0 > 1$, write
\[
H = \sum_{\substack{d < \sqrt{D_0}\\d|P}} h(d).
\]
Then, \cite[Theorem 6.4]{ANTKowalski} says that
\begin{equation}
\label{eq:Selberg}
\sum_{(n,P) = 1} a_n \leq XH^{-1} + O\lp \sum_{\substack{d \leq D_0\\d|P}} \tau_3(d) r_d \rp.
\end{equation}
We now explain how to apply \eqref{eq:Selberg} in our context of orbit-counting.
We will typically work in a subset $W \subset V(\Z)$ (e.g. the main body, the 
cusp...), and we will suppose we have a set $S \subset W$ which satisfies $S = \cap_p S_p$,
where for each prime $p$, the set $S_p$ is defined by congruence conditions
modulo $p$. We wish to estimate $N(S,X)$, which will generally be some 
orbit-counting function of $S$ inside $W$ (to be made precise in future applications).
Let $T_p$ be the complement of $S_p$ in $W$, and fix a number
$z < X$. Let $P(z) = \prod_{p < z} p$, and for a number $d | P(z)$ set
\[
a_d = N\lp\bigcap_{p|d}T_p \bigcap_{p | \frac{P(z)}{d}} S_p,X\rp.
\]
If $d \nmid P(z)$, set $a_d = 0$.
To apply the Selberg sieve, we need an estimate like \eqref{eq:SelbergHyp}.
Let $L$ be a translate of $mW$, for some squarefree $m \in \Z$, and
assume that we have an estimate of the form
\begin{equation}
\label{eq:powerSave}
N(L,X) = km^{-A}X^B + O(m^{-A+C}X^{B-D}),
\end{equation}
for some non-negative constants $A,B,C,D$ and $k$. Then, it follows that
\[
\sum_{n \equiv 0 \bmod{d}} a_n = N(\cap_{p \mid d} T_p,X) = kg_d X^B + r_d,
\]
where for a prime $p$, the quantity $g_p$ is the density of $T_p$,
for $d$ squarefree we set $g_d = \prod_{p \mid d} g_p$, and we have $r_d = O(d^Cg_dX^{B-D})$.
Then, by \eqref{eq:Selberg}, we have
\[
a_1 = \sum_{(n,P(z)) = 1} a_n \leq kX^B H^{-1} + O\lp \sum_{\substack{d \leq D_0\\d|P}} \tau_3(d) r_d \rp.
\]
Assume now that as $p \to \infty$, the density $g_p$ converges to some constant
$\lambda \in (0,1)$. Then, we are able to obtain bounds for $H$ and $r_d$
depending only on $X$ and the choice of $D$. Given that $d^{-\eps} \ll_{\eps} g_d \ll_{\eps} d^{\eps}$,
we get that $H = D_0^{1/2+o(1)}$. For the error term, we get that
\[
\left| \sum_{\substack{d \leq D_0\\d|P}} \tau_3(d) r_d \right| \ll_{\eps} X^{B-D}D_0^{\eps}\sum_{d\leq D_0} d^C \ll_{\eps} X^{B-D} D_0^{C+1+\eps}.
\]
The end result is that $a_1 \ll_{\eps} X^{B}D_0^{-1/2+\eps} + X^{B-D}D_0^{C+1+\eps}$.
By making an appropriate choice of $D$ as a power of $X$, we can optimize this
expression to yield $a_1 = O(X^{B-\delta})$ for some $\delta > 0$.

So, in summary, to use the Selberg sieve in the same way that is used
in \cite{STSelberg}, it will suffice to have an expression of the form 
\eqref{eq:powerSave}, and a proof that the densities of our sets $S_p$ converge
to some constant in $(0,1)$ as $p$ goes to infinity.

\subsection{Reductions}
We return to the setting of Section \ref{subs:averaging}, where we had
\[
N(V_i,X) = \frac{1}{r_i \vol(G_0)} \int_{g \in \cF} \# \{v \in V(\Z)^{red} \cap (g G_0 \kappa(\Lambda L_i))_{< X} \} dg. 
\]
To estimate this quantity, we will make some necessary reductions. We will
begin with a ``cutting-off-the-cusp'' result, which amounts to saying that
not too many points in the cusp are irreducible.
\begin{prop}
\label{prop:cusp}
Let $v_0$ be the coefficient of the highest weight in $V$. Then, there
exists a constant $\delta_1 > 0$ such that
\[
\int_{g \in \cF} \# \{v \in (V(\Z) \setminus W_0(\Z)) \cap g\cB_X \mid v_0 = 0 \} dg = O(X^{\dim V - \delta_1}).
\]
\end{prop}
\begin{proof}
This is the content of \cite[Proposition 8.12]{LagaThesis}.
\end{proof}
In a similar spirit, we also show that most of the elements in the main
body are irreducible:
\begin{prop}
Let $v_0$ be the coefficient of the highest weight in $V$. Then, there
exists a constant $\delta_2 > 0$ such that
\[
\int_{g \in \cF} \# \{v \in V(\Z)^{red} \cap g\cB_X \mid v_0 \neq 0 \} dg = O(X^{\dim V - \delta_2})
\]
\end{prop}
\begin{proof}
We will prove this statement by using the Selberg sieve, as explained
in Section \ref{subs:Selberg}. First of all, if $v \in V(\Z)$ is reducible, 
then for all primes $p$ not dividing the $N$ fixed in Proposition \ref{prop:admissible}
the reduction of $v$ mod $p$ is reducible, since by Theorem \ref{theo:intKostant}
$v$ is $G(\Z_p)$ conjugate to $\kappa_b$.
By \cite[Proof of Lemma 8.22]{LagaThesis}, the density of elements in $V(\F_p)$ which
are $\F_p$-reducible converges to some constant $\lambda \in (0,1)$.

To apply the Selberg sieve, we need some result in the style of \eqref{eq:powerSave}.
This is essentially the content of \cite[Proposition 8.15 and Theorem 8.17]{LagaThesis};
a power saving estimate can be obtained similarly to \cite[Proposition 10.5]{BG},
and the contribution from the congruence conditions can be done similarly
to our proof of Theorem \ref{theo:congruence}: we do not repeat it here
for the sake of concision.
\end{proof}

\subsection{Counting reducible orbits}
\label{subs:counting}
The previous reductions show that when trying to estimate
\[
N(V_i,X) = \frac{1}{r_i \vol(G_0)} \int_{g \in \cF} \# \{v \in V(\Z)^{red} \cap (g G_0 \kappa(\Lambda L_i))_{< X} \} dg,
\]
it is sufficient to work over the cusp $W_0(\Z)$ up to a power-saving
error term.
Given that $\cF$ is a box-shaped fundamental domain,
we can write it as a disjoint union $\cF' \cup \omega T_c \ol{K}$, where $\omega$,
$c$ and $\ol{K}$ are as in Section \ref{subs:corners} and $\cF'$ is a subset of
\[
\omega \cdot \{t \in T^{\theta}(\R)^{\circ} \mid \alpha(t) \leq c \text{ for some }\alpha \in S_G\} \cdot \ol{K}.
\]
An explicit computation (e.g. following the reasoning in this section and
in Section \ref{subs:cases}) shows that the integral in \eqref{eq:averaging} is negligible when $\cF$
is substituted by $\cF'$. Hence, it suffices to integrate over $\omega T_c \ol{K}$.
Given that $G_0$ is $K$-invariant and that $dk$ can be normalised so that
$\ol{K}$ has volume 1, we get:
\[
N(V_i,X) = \frac{1}{r_i \vol(G_0)}\int_{n \in \omega} \int_{t \in T_c} \# \{v \in W_0(\Z) \cap nt\cB_X \} \delta^{-1}(t) dn d^{\times}t + O\lp X^{\dim V - \delta} \rp
\]
for some $\delta > 0$. It would be desirable to estimate the lattice
points in the region using Davenport's lemma; however, as noted in \cite{SSSVcusp}, the cuspidal region is too skewed
to apply the lemma directly: in particular, some of volumes of the projections
can be of the order of the main term. To circumvent this, we will ``slice''
the region $W_0(\Z)$ according to the values of the height-one coefficients.
For $v \in W_0(\Z)$, denote by $(\sigma_1(v),\dots,\sigma_r(v))$ its
height-one coefficients. Then, for any $b = (b_1,\dots,b_r) \in \R^r$
and any subset $S \subset W_0(\R)$, we will denote
\[
S|_b = \{v \in S \mid (\sigma_1(v),\dots,\sigma_r(v)) = b\}.
\]
Then, we can express
\[
\# (W_0(\Z) \cap (nt\cB_{X})) = \sum_{b \in \Z^{r}} \# (W_0(\Z) \cap (nt\cB_{X})|_b).
\]
Actually, we can assume that in the sum over $b = (b_1,\dots,b_r)\in \Z^r$, none
of the components $b_i$ are equal to zero due to the following:
\begin{lemma}
Let $v \in W_0(\R)$. If $\sigma_i(v) = 0$ for some $i$, then $\Delta(v) = 0$.
\end{lemma}
\begin{proof}
Let $\{\alpha_1,\dots,\alpha_k\}$ be the height-one weights, and assume
that the coefficient of $\alpha_i$ of $v$ is zero. Let 
$\lambda_i\colon \G_m \to G_{\C}$ be the one-parameter subgroup 
such that $(\alpha_j \circ \lambda_i)(t) = t^{\delta_{ij}}$. Then,
$v$ has no positive weights with respect to $\lambda_i$, and so by Proposition
\ref{prop:H-M} we get the result.
\end{proof}
When applying Proposition \ref{prop:Davenport} to $(nt\cB_{X})|_b$,
we get
\begin{equation}
\label{eq:Davenport}
\# (W_0(\Z) \cap (nt\cB_{X})|_b) = \vol((nt\cB_{X})|_b) (1 + O(X^{-1})).
\end{equation}
The term $O(X^{-1})$ can be obtained as follows: each coefficient $v_0$ in $W_0(\R)$
has a weight under the action of $T$, which we will denote $w(v_0)$. When
performing the slicing, that is, fixing the values of the height-one coefficients,
all the weights turn out to be $\gg X$, and given that the volume 
of the region is the product of the weights of the different coordinates,
we obtain the saving of size $X$.

Given that unipotent transformations preserve the volume, and that we can
normalise $dn$ so that $\vol(\omega) = 1$, we can write the following:
\begin{equation}
\label{eq:middlestep}
N(V_i,X) = \frac{1}{r_i \vol(G_0)} \sum_{b \in (\Z \setminus\{0\})^r}\int_{t \in T_c} \vol((t\cB_X)|_b) \delta^{-1}(t) d^{\times}t + O(X^{\dim V - \delta}).
\end{equation}
For each height-one coefficient $v_i$, we will
denote $\beta_i := (X w(v_i) b_i)^{-1}$, and $\beta = (\beta_i)_i$. Denote
by $W_{\flat}$ the set of coordinates of $W_0$ of non-positive height. It follows that
\[
\vol((t\cB_X)|_b) = \vol(t X\cdot\cB|_{\beta}) = X^{\dim W_{\flat}} \prod_{v \in W_{\flat}}w(v) \vol(\cB|_{\beta}).
\] 
We will make the change of variables $t \mapsto \beta = (\beta_1,\dots,\beta_r)$, under which
$d^{\times}t = d^{\times} \beta = \prod_i\frac{d \beta_i}{\beta_i}$. In
Section \ref{subs:cases}, we will explicitly compute the volume of the cuspidal region for
each of the possible cases. We will obtain a polynomial $Z(\beta) = \prod_{i}\beta_i^{e_i}$
with integer exponents $e_i \geq 2$, and we will see that
\begin{equation}
\label{eq:explicit}
X^{\dim W_{\flat}} \prod_{v \in W_{\flat}}w(v) \delta^{-1}(t) = X^{\dim V} \frac{Z(\beta)}{Z(b)}.
\end{equation}
It follows that
\[
\int_{t \in T_c} \vol((t\cB_X)|_b)  \delta^{-1}(t) d^{\times}t = \frac{X^{\dim V}}{|Z(b)|} \int_{\beta \in \R_{>0}^r \setminus T'} Z(\beta) \vol(\cB|_{\beta}) d^{\times}\beta,
\]
where $T'$ is the region corresponding to $T(\R) \setminus T_c$. It is not difficult to see
that the integral over $T'$ is $O(X^{-1})$, and hence can be
added to the error term. For an element $v \in W_0(\Z)$, define $Z(v) := Z(\sigma_i(v)) = \prod_i \sigma_i(v)^{e_i}$
and $Z^{\times}(v) :=  \prod_i \sigma_i(v)^{e_i-1}$. Then,
\begin{align}
\label{eq:laststep}
\begin{split}
\int_{t \in T_c} \vol((t\cB_X)|_b) \delta^{-1}(t) d^{\times}t &= \frac{X^{\dim V}}{|Z(b)|} \int_{\beta \in \R_{>0}^r} Z(\beta) \vol(\cB|_{\beta}) d^{\times}\beta + O(X^{\dim V - 1}) \\
&= \frac{X^{\dim V}}{|Z(b)|} \int_{v \in \cB \cap W_0(\R)_+} Z^{\times}(v) dv + O(X^{\dim V - 1}).
\end{split}
\end{align}
Here, $W_0(\R)_+ = \{v \in W_0(\R) \mid \sigma_i(v) > 0, \, \forall i\}$.
Combining \eqref{eq:middlestep} and \eqref{eq:laststep},
and summing over all $b$, we obtain that
\begin{equation}
\label{eq:Vi}
N(V_i,X) = 2^r \prod_{i=1}^r \zeta(e_i) \cdot \lp \frac{1}{r_i \vol(G_0)} \int_{v \in \cB \cap W_0(\R)_+} Z^{\times}(v) dv \rp X^{\dim V} + O(X^{\dim X - \delta}).
\end{equation}
To obtain the desired asymptotic for $N(V(\Z)^{red},X)$, it suffices to
use the inclusion-exclusion principle. For any subset $I \subset \{1,\dots,r\}$,
the same procedure as above obtains \eqref{eq:Vi} for the the set
$V_I = \cap_{i\in I} V_i$, with the appropriate constants substituted.
This concludes the proof of Theorem \ref{theo:reducible}.

\subsection{Computing the constant}
\label{subs:constant}
As promised, we will compute the constant of the main term of Theorem 
\ref{theo:reducible}. We will do so using a Jacobian change-of-variables
formula, whose statement and proof are completely analogous to \cite[Proposition 14]{SSSVcusp}:
we include the proof in our case for convenience.
\begin{prop}
\label{prop:Jacobian}
Let $\phi \colon W_0(\R) \to \R$ be a measurable function. Then, there 
exists a non-zero rational constant $\cJ \in \Q^{\times}$ such that
\[
\int_{v \in W_0(\R)} \phi(v) |Z^{\times}(v)| dv = |\cJ| \int_{\substack{b \in B(\R) \\ \Delta(b) \neq 0}} \lp \sum_{v \in \frac{\pi^{-1}(b)}{P^{-,\theta}(\R)}} \int_{h \in P^{-,\theta}(\R)}\phi(h \cdot v) dh \rp db.
\]
Here, $dv$ and $db$ are Euclidean measures, and $dh = \delta^{-1}(t) dn d^{\times}t$
is a right Haar measure for $P^{-,\theta}(\R)$.
\end{prop}
\begin{proof}
Let $U \subset B(\R)$ be an open subset, and let $\sigma \colon U \to W_0(\R)$
be a continuous section of the GIT quotient map $\pi \colon V \to B$.
We first claim that we have
\begin{equation}
\label{eq:CVstep}
\int_{v \in P^{-,\theta}(\R)\sigma(U)} \phi(v) |Z^{\times}(v)| dv = |\cJ| \int_{b \in U} \int_{h \in P^{-,\theta}(\R)}\phi(h \cdot \sigma(b)) dh db
\end{equation}
for some non-zero rational constant $\cJ$. By the Stone-Weierstrass theorem,
we can assume that $\sigma$ is piecewise analytic, in which case we have
\[
\int_{v \in P^{-,\theta}(\R)\sigma(U)} \phi(v) |Z^{\times}(v)| dv = \int_{b \in U} \int_{h \in P^{-,\theta}(\R)}|\cJ_{\sigma}(h,b)|\phi(h \cdot \sigma(b)) dh db,
\]
where $\cJ_{\sigma}(h,b)$ denotes the determinant of the Jacobian matrix
arising from the change of variables that takes the measure $Z^{\times}(v) dv$
to $dh db$. We will now show that $\cJ_{\sigma}(h,b)$ is independent of
$\sigma$, $h$ and $b$.

To show that $\cJ_{\sigma}(h,b)$ is independent of $h$, we fix $\gamma \in P^{-,\theta}(\R)$
and consider the change of variables $v \mapsto \gamma \cdot v$ in $W_0(\R)$.
We have that $Z^{\times}(\gamma \cdot v) d(\gamma \cdot v) = \chi(\gamma) Z^{\times}(v) dv$
for some character $\chi \colon P^{-,\theta}(\R) \to \R_{> 0}$, which we now
determine explicitly. If $\gamma \in \ol{N}(\R)$, then $\chi(\gamma) = 1$,
since neither $Z^{\times}$ or the volume of $W_0(\R)$ are changed by the
action of $\ol{N}(\R)$. Now, assume that $\gamma \in T^{\theta}(\R)$. On
one hand, we have that
\[
Z^{\times}(\gamma \cdot v) = \lp\prod_{\alpha_i \in S_G} \alpha_i(\gamma)^{e_i-1}\rp Z^{\times}(v).
\]
On the other hand, we have that
\[
d(\gamma \cdot v) = \lp\prod_{\substack{\alpha \in \Phi(G,T^{\theta})\\ \Ht(\alpha) \leq 1}} \alpha(\gamma)\rp^{-1} dv.
\]
In view of \eqref{eq:explicit}, we conclude that $\chi(\gamma) = \delta^{-1}(\gamma)$.
On the other hand, for $\gamma = nt \in P^{\theta}(\R)$ we also have that
\[
\cJ_{\sigma}(\gamma h,b) d(\gamma h) db = \delta^{-1}(t) \cJ_{\sigma}(\gamma h,b) dh db
\]
because $dh$ is a right Haar measure of $H$, and $\delta^{-1}(t)$ is the
corresponding modular function (cf. \cite[(8.26)]{Knapp}). We then have that
\[
\cJ_{\sigma}(\gamma h, b)d(\gamma h) db = Z^{\times}(\gamma v) dv = \delta^{-1}(t) Z^{\times}(v) dv = \delta^{-1}(t) \cJ_{\sigma}(h,b) dhdb,
\]
and hence that $\cJ_{\sigma}(h,b) = \cJ_{\sigma}(\gamma h,b)$ is independent of $h$, as wanted.

The rest of the proof now follows analogously to \cite[Proof of Proposition 3.10]{BSquartics}.
More precisely, that $\cJ_{\sigma}(h,b)$ is independent of $\sigma$ is analogous
to Step 2 in \cite[Proof of Proposition 3.10]{BSquartics}; in particular,
we can take $\sigma$ to be the Kostant section. Then, independence of
$b$ follows from steps 3 and 4 in \cite[Proof of Proposition 3.10]{BSquartics}.

Thus, we have shown \eqref{eq:CVstep}. The proposition now follows from
from \eqref{eq:CVstep} in a similar way as how \cite[Proposition 3.7]{BSquartics}
follows from \cite[Proposition 3.10]{BSquartics}.
\end{proof}
The proof of Theorem \ref{theo:reducible} shows that the leading constant
in the asymptotic for $N(V_i,X)$ is:
\[
2^r \prod_{i=1}^r \zeta(e_i) \sum_{i=1}^k \frac{1}{r_i \vol(G_0)} \int_{v \in \cB \cap W_0(\R)_+} Z^{\times}(v) dv.
\]
Here, $r$ is the amount of height-one coefficients, $\zeta$ is the Riemann
zeta function, and $e_i$ are the exponents corresponding to $Z(\beta) = \prod_i \beta_i^{e_i}$.
We will now give a more succint description of the above integral, following
\cite[\S 4.3]{SSSVcusp}.

Given that $G_0$ is $K$-invariant, we can write it as $G_0 = \cS\cdot K$,
for some set $\cS \subset P^{-,\theta}(\R)$. We have the following lemma:
\begin{lemma}
The map $\pi \colon K \kappa(\Lambda L_i) \cap W_0(\R)_+ \to L_i$ is $r_i$ to $1$.
\end{lemma}
\begin{proof}
The result follows from the fact that every element $v \in K\kappa(\Lambda L_i)$
satisfies $\#\Stab_{G(\R)} v = r_i$, and that if $g \in \Stab_{G(\R)} v$,
then writing $g = pk$ for $p \in P^{-,\theta}(\R)$ and $k \in K$, we get
that $kv = p^{-1}v$, so $kv$ belongs to $K \kappa(\Lambda L_i) \cap W_0(\R)_+$.
Conversely, given that $P^{-,\theta}$ acts simply transitively on $W_0(\R)$,
any element in $K \kappa(L_i) \cap W_0(\R)_+$ that is conjugate to $v$
has to be of the form $kv = p'v$ for some $k \in K$ and $p' \in P^{-,\theta}$.
\end{proof}

Now, setting $\phi$ to be the indicator function of $\cB \cap W_0(\R)_+$
in Proposition \ref{prop:Jacobian}, we obtain
\[
\frac{1}{r_i \vol(G_0)} \int_{v \in \cB_{<X} \cap W_0(\R)_+} Z^{\times}(v) dv
= \frac{|\cJ|r_i \vol(K\cS) \vol(\{b \in \Lambda L_i \mid \Ht(b) < 1\})}{r_i \vol(\cS K)}.
\]
However, we observe that:
\begin{lemma}
We have $\cS K = K\cS$.
\end{lemma}
\begin{proof}
Recall that $G_0$ is left and right 
$K$-invariant and satisfies $G_0^{-1} = G_0$. Then,
\[
K \cS \subset K \cS K = \cS K = G_0 = G_0^{-1} = K\cS^{-1}.
\]
By uniqueness in the Iwasawa decompsition, we must have that $\cS \subset \cS^{-1}$,
and symmetrically that $\cS = \cS^{-1}$. Therefore, $\cS K = K \cS^{-1} = K \cS$,
as wanted.
\end{proof}
We are left to deal with the volumes of the corresponding $L_i$ terms, which
we do using the inclusion-exclusion principle. The end result is
\[
N(V(\Z)^{red},X) \sim 2^r \prod_{i=1}^r \zeta(e_i) |\cJ| \vol(\{b \in B(\R) \mid \Ht(b) < 1\}) X^{\dim V}.
\]
We can compare this result with the asymptotics for $N(V(\Z)^{irred},X)$,
which can be read off \cite[Theorem 8.8]{LagaThesis}. In there, one of
the factors of the constant is related to the volume of $G(\Z)\backslash G(\R)$ 
with respect to a suitably normalised Haar measure, and can be done 
following \cite{LVolume} and \cite{Coleman58}, for instance. Surprisingly,
we get that
\[
\vol(G(\Z) \backslash G(\R)) = c \prod_{i=1}^r \zeta(e_i),
\]
where $c$ is the order of the fundamental group of $G_{\C}$ and $e_i$
turn out to be the same exponents as above; in particular, the constants for the
reducible and irreducible case appear to be the same up to some rational
factor, thus answering Question 2 of \cite{SSSVcusp} affirmatively for
our representations $(G,V)$. However, the two methods of obtaining the
constants appear to be fundamentally different, and we wonder if there is
any ``natural'' explanation as to why they should give the same result.

\subsection{Case-by-case analysis}
\label{subs:cases}

In this section, we complete the proof of Theorem \ref{theo:reducible} by performing a case-by-case analysis.
For the $D_n$ and $E_n$ cases, we will explicitly compute the dimension
and volume of $W_{\flat}$ (which was defined to be the set of coefficients
of $W_0$ of non-positive height), and the modular function
$\delta(t) = \prod_{\beta \in \Phi_G^{-}} \beta(t) = \det \Ad(t)|_{\Lie \overline{N}(\R)}$.

\subsubsection{$D_{2n+1}$}

The exposition in the $D_n$ cases is inspired by \cite[Appendix A]{LagaThesis}
and \cite[\S 7.2.1]{ShankarD2n+1}.
We start by describing explicitly the representation $(G,V)$ of $D_{2n+1}$
in the form given by Table \ref{table:explicit}.

Let $n \geq 2$ be an integer. Let $U_1$ be a $\Q$-vector space with basis
$\{e_1,\dots,e_n,u_1,e_n^{*},\dots,e_1^{*}\}$, endowed with the symmetric
bilinear form $b_1$ satisfying $b_1(e_i,e_j) = b_1(e_i,u_1) = b_1(e_i^*,e_j^*) = b_1(e_i^*,u_1) = 0$,
$b_1(e_i,e_j^*) = \delta_{ij}$ and $b_1(u_1,u_1) = 1$ for all $1 \leq i,j \leq n$.
In this case, given a linear map $A \colon U \to U$ we can define its \emph{adjoint}
as the unique map $A^{\ast} \colon U \to U$ satisfying $b_1(Av,w) = b_1(v,A^{\ast}w)$
for all $v,w \in U$. In terms of matrices, $A^*$ corresponds to taking 
the reflection of $A$ along its antidiagonal when working with the fixed basis. We can define 
$\So(U_1,b_1) := \{ g \in \SL(U_1) \mid gg^* = \id\}$, with a Lie algebra
that can be identified with $\{A \in \End(U) \mid A = -A^* \}$.

Let $U_2$ be a $\Q$-vector space with basis $\{f_1,\dots,f_n,u_2,f_n^{\ast},\dots,f_1^{\ast}\}$,
with a similarly defined bilinear form $b_2$. Let $(U,b) = (U_1,b_1) \oplus (U_2,b_2)$.
Let $H = \So(U,b)$, and consider $\hh := \Lie H$. With respect to the basis
\[
\{e_1,\dots,e_n,u_1,e_n^{\ast},\dots,e_1^{*},f_1,\dots,f_n,u_2,f_n^{\ast},\dots,f_1^{\ast}\},
\]
the adjoint of a block matrix according to the bilinear form $b$ is given by
\[
\begin{pmatrix}
A & B \\ C & D
\end{pmatrix}^*
=
\begin{pmatrix}
A^* & C^* \\ B^* & D^*
\end{pmatrix},
\]
where $A^*,B^*,C^*,D^*$ denote reflection by the antidiagonal. An element
of $\hh$ is given by
\[
\left\{\begin{pmatrix}
B & A \\ -A^* & C
\end{pmatrix} \,\middle\vert\; B = -B^*, \, C = -C^*\right\}.
\]
The stable involution $\theta$ is given
by conjugation by $\diag(1,\dots,1,-1,\dots,-1)$, where the first $2n+1$
entries are $1$ and the last $2n+1$ entries are given by $-1$. Under this
description, we see that
\[
V = \left\{\begin{pmatrix}
0 & A \\ -A^* & 0
\end{pmatrix} \,\middle\vert\; A \in \text{Mat}_{(2n+1)\times(2n+1)}\right\}.
\]
Moreover, $G = (H^{\theta})^{\circ}$ is isomorphic to $\So(U_1) \times \So(U_2)$.
We will use the map
\[
\begin{pmatrix}
0 & A \\ -A^* & 0
\end{pmatrix} \mapsto A
\]
to establish a bijection between $V$ and $\Hom(U_2,U_1)$, where $(g,h) \in \So(U_1) \times \So(U_2)$
acts on $A \in V$ as $(g,h)\cdot A = gAh^{-1}$.

Let $T$ be the maximal torus $\diag(t_1,\dots,t_n,1,t_n^{-1},\dots,t_1^{-1},s_1,\dots,s_n,1,s_n^{-1},\dots,s_1^{-1})$
of $G$. A basis of simple roots for $G$ is
\[
S_G = \{t_1-t_2,\dots,t_{n-1}-t_n\} \cup \{s_1-s_2,\dots,s_{n-1}-s_n\}.
\]
A positive root basis for $V$ can be taken to be
\[
S_V = \{t_1-s_1,s_1-t_2,\dots,t_n-s_n,s_n\}.
\]
For convenience, we now switch to multiplicative notation for the roots.
We make the change of variables $\alpha_i = t_i/t_{i+1}$ for $i = 1,\dots,n-1$
and $\alpha_n = t_n$; similarly $\gamma_i = s_i/s_{i+1}$ for $i = 1,\dots,n-1$
and $\gamma_n = s_n$. The estimate for the volume of $W_{\flat}$ becomes:
\[
\prod_{v \in W_{\flat}} Xw(v) = X^{2n^2+2n+1} \prod_{i=1}^n \alpha_i^{-2in+i^2-2i}\gamma_i^{-2in+i^2}.
\]
The modular function in our case is
\[
\delta^{-1}(t) = \prod_{i=1}^n \alpha_i^{2in-i^2}\gamma_i^{2in-i^2}.
\]
Changing variables to $\beta_i = (X w(v_i)b_i)$, where $v_i$ are the 
height-one coefficients, we obtain
\[
\prod_{v \in W_{\flat}} Xw(v) \delta^{-1}(t) = X^{4n^2+4n+1} \frac{Z(\beta)}{Z(b)},
\]
where $Z(\beta) := \prod_{i=1}^n (\beta_{2i-1}\beta_{2i})^{2i}$.
\subsubsection{$D_{2n}$}
The analysis in this case is very similar to the $D_{2n+1}$ case. Now, we
consider the $\Q$-vector space $U_1$ with basis $\{e_1,\dots,e_n,e_n^*,\dots,e_1^*\}$,
endowed with a symmetric bilinear form $b_1(e_i,e_j) = b_1(e_i^*,e_j^*) = 0$, 
$b_1(e_i,e_j^*) = \delta_{ij}$. We also consider a $\Q$-vector space $U_2$
with basis $\{f_1,\dots,f_n,f_n^*,\dots,f_1^*\}$, with an analogous symmetric
bilinear form $b_2$.

Let $(U,b) = (U_1,b_1) \oplus (U_2,b_2)$, let $H' = \So(U,b)$ and define
$H$ to be the quotient of $H'$ by its centre of order 2. Under the basis
\[
\{e_1,\dots,e_n,e_n^*,\dots,e_1^*,f_1,\dots,f_n,f_n^*,\dots,f_1^*\},
\]
the stable involution is given by conjugation with $\diag(1,\dots,1,-1,\dots,-1)$.
Similarly to the $D_{2n+1}$ case, we have
\[
V = \left\{\begin{pmatrix}
0 & A \\ -A^* & 0
\end{pmatrix} \,\middle|\; A \in \text{Mat}_{2n \times 2n}\right\},
\]
where $A^*$ denotes reflection by the antidiagonal. In this case, the
group $G = (H^{\theta})^{\circ}$ is isomorphic to $\So(U_1) \times \So(U_2)/\Delta(\mu_2)$,
where $\Delta(\mu_2)$ denotes the diagonal inclusion of $\mu_2$ into the
centre $\mu_2 \times \mu_2$ of $\So(U_1) \times \So(U_2)$. As before,
we can identify $V$ with the space of $2n\times2n$ matrices using the
map
\[
\begin{pmatrix}
0 & A \\ -A^* & 0
\end{pmatrix} \mapsto A,
\]
where $(g,h) \in G$ acts by $(g,h)\cdot A = gAh^{-1}$.

We consider the maximal torus $T$ of $H$ given by $\diag(t_1,\dots,t_n,t_n^{-1},\dots,t_1^{-1},s_1,\dots,s_n,s_n^{-1},\dots,s_1^{-1})$.
A basis of simple roots for $H$ and $G$ are given by
\begin{align*}
&S_H = \{t_1-s_1,s_1-t_2,\dots,s_{n-1}-t_n,t_n-s_n,s_n+t_n\}, \\
&S_G = \{t_1-t_2,\dots,t_{n-1}-t_n,t_{n-1}+t_n\} \cup \{s_1-s_2,\dots,s_{n-1}-s_n,s_{n-1}+s_n\}.
\end{align*}

%The space $W_0$ is given by...
Let $\alpha_i = t_i/t_{i+1}$ and $\gamma_i = s_i/s_{i+1}$ for $i = 1,\dots,n$,
and let $\alpha_n = t_{n-1}t_n$ and $\gamma_n = s_{n-1}s_n$. Under this change
of variables, the volume of $W_{\flat}$ is:
\[
\prod_{v \in W_{\flat}}Xw(v) = X^{2n^2} \left(\prod_{i=1}^{n-2}\alpha_i^{-2in+i^2-i}\alpha_{n-1}^{(-n^2-n+4)/2}\alpha_n^{(-n^2-n)/2}\prod_{i=1}^{n-2}\gamma_i^{-2in+i^2+i}(\gamma_{n-1}\gamma_n)^{(-n^2+n)/2}\right).
\]
The modular function is
\[
\delta^{-1}(t) = \prod_{i=1}^{n-2} (\alpha_i\gamma_i)^{i^2-2in+i}(t) (\alpha_{n-1}\gamma_{n-1}\alpha_n\gamma_n)^{-(n-1)n/2}(t).
\]
As before, we can compute:
\[
\prod_{v \in W_{\flat}} Xw(v) \delta^{-1}(t) = X^{4n^2} \frac{Z(\beta)}{Z(b)},
\]
where $Z(\beta) = \prod_{i=1}^{n-1}(\beta_{2i-1}\beta_{2i})^{2i}\cdot(\beta_{2n-1}\beta_{2n})^n$.
\subsubsection{$E_6$}
For the $E_6$ case, we use the conventions and computations in \cite[\S 2.3, \S 5]{ThorneE6}.

Let $S_H = \{\alpha_1,\dots,\alpha_6\}$, where the Dynkin diagram of $H$ is:
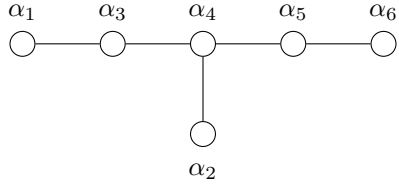
\begin{center}
\begin{tikzpicture}[scale=0.6]
  % Nodes
  \node[circle, draw,label={$\alpha_1$}] (A1) at (0,0) {};
  \node[circle, draw,label={$\alpha_3$}] (A3) at (2,0) {};
  \node[circle, draw,label={$\alpha_4$}] (A4) at (4,0) {};
  \node[circle, draw,label={$\alpha_5$}] (A5) at (6,0) {};
  \node[circle, draw,label={$\alpha_6$}] (A6) at (8,0) {};
  \node[circle, draw,label=below:{$\alpha_2$}] (A2) at (4,-2) {};
  
  % Edges
  \draw[-] (A1) -- (A3) -- (A4) -- (A5) -- (A6);
  \draw[-] (A2) -- (A4);
\end{tikzpicture}
\end{center}
The pinned automorphism $\vartheta$ consists of a reflection
around the vertical axis. We can define a root basis $S_G = \{\gamma_1,\gamma_2,\gamma_3,\gamma_4\}$
of $G$ as $\gamma_1 = \alpha_3+\alpha_4$, $\gamma_2 = \alpha_1$, $\gamma_3 = \alpha_3$ and $\gamma_4 = \alpha_2 + \alpha_4$.
Under this basis, we have
\[
\prod_{v \in W_{\flat}}X\omega(v) = X^{22} (\gamma_1^{-12}\gamma_2^{-18}\gamma_3^{-22}\gamma_4^{-12})
\]
The modular function is
\[
\delta^{-1}(t) = \left(\gamma_1^{8}\gamma_2^{14}\gamma_3^{18}\gamma_4^{10}\right)(t).
\]
The weights of the height-one coefficients are $\{\gamma_2,-\gamma_1+\gamma_3+\gamma_4,\gamma_3,\gamma_1-\gamma_3\}$.
In light of this, we obtain
\[
\prod_{v \in W_{\flat}}X\omega(v) \delta^{-1}(t) = X^{42}\frac{Z(\beta)}{Z(b)}.
\]
where $Z(\beta) = \beta_1^4 \beta_2^2 \beta_3^8 \beta_4^6$.
\subsubsection{$E_7$}
For the $E_7$ and $E_8$ cases, we follow the conventions in \cite{RTE78}.
Let $S_H = \{\alpha_1,\dots,\alpha_7\}$, where the Dynkin diagram of $H$ is:
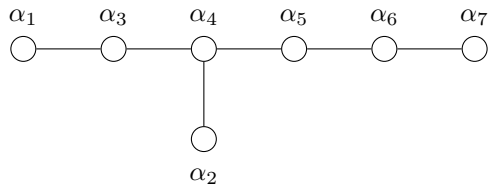
\begin{center}
\begin{tikzpicture}[scale=0.6]
  % Nodes
  \node[circle, draw,label={$\alpha_1$}] (A1) at (0,0) {};
  \node[circle, draw,label={$\alpha_3$}] (A3) at (2,0) {};
  \node[circle, draw,label={$\alpha_4$}] (A4) at (4,0) {};
  \node[circle, draw,label={$\alpha_5$}] (A5) at (6,0) {};
  \node[circle, draw,label={$\alpha_6$}] (A6) at (8,0) {};
  \node[circle, draw,label={$\alpha_7$}] (A7) at (10,0) {};
  \node[circle, draw,label=below:{$\alpha_2$}] (A2) at (4,-2) {};
  
  % Edges
  \draw[-] (A1) -- (A3) -- (A4) -- (A5) -- (A6) -- (A7);
  \draw[-] (A2) -- (A4);
\end{tikzpicture}
\end{center}
The root basis $S_G = \{\gamma_1,\dots,\gamma_7\}$ can be described as
\begin{align*}
\gamma_1 &= \alpha_3 + \alpha_4 \\
\gamma_2 &= \alpha_5 + \alpha_6 \\
\gamma_3 &= \alpha_2 + \alpha_4 \\
\gamma_4 &= \alpha_1 + \alpha_3 \\
\gamma_5 &= \alpha_4 + \alpha_5 \\
\gamma_6 &= \alpha_6 + \alpha_7 \\
\gamma_7 &= \alpha_2 + \alpha_3 + \alpha_4 + \alpha_5 \\
\end{align*}
The volume of $W_{\flat}$ can be computed to be
\[
\prod_{v \in W_{\flat}} Xw(v) = X^{35} (\gamma_1^{-15/2}\gamma_2^{-13}\gamma_3^{-33/2}\gamma_4^{-18}\gamma_5^{-35/2}\gamma_6^{-15}\gamma_7^{-21/2}).
\]
The modular function for $G$ can be computed to be
\[
\delta^{-1}(t) = (\gamma_1^7\gamma_2^{12}\gamma_3^{15}\gamma_4^{16}\gamma_5^{15}\gamma_6^{12}\gamma_7^7)(t).
\]
We can compute the weights $\beta_i$ corresponding to the height-one
coefficients, with the end result being
\[
\prod_{v \in W_{\flat}} Xw(v)\delta^{-1}(t) = X^{70}\frac{Z(\beta)}{Z(b)},
\] 
for $Z(\beta) = \beta_1^2 \beta_2^5 \beta_3^6 \beta_4^8 \beta_5^7 \beta_6^4 \beta_7^3$.
\subsubsection{$E_8$}
Let $S_H = \{\alpha_1,\dots,\alpha_8\}$, where the Dynkin diagram of $H$ is:
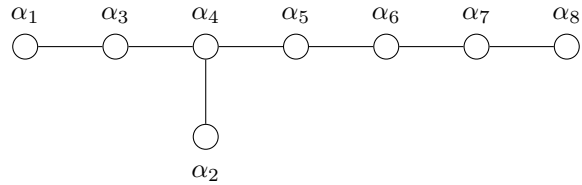
\begin{center}
\begin{tikzpicture}[scale=0.6]
  % Nodes
  \node[circle, draw,label={$\alpha_1$}] (A1) at (0,0) {};
  \node[circle, draw,label={$\alpha_3$}] (A3) at (2,0) {};
  \node[circle, draw,label={$\alpha_4$}] (A4) at (4,0) {};
  \node[circle, draw,label={$\alpha_5$}] (A5) at (6,0) {};
  \node[circle, draw,label={$\alpha_6$}] (A6) at (8,0) {};
  \node[circle, draw,label={$\alpha_7$}] (A7) at (10,0) {};
  \node[circle, draw,label={$\alpha_8$}] (A8) at (12,0) {};
  \node[circle, draw,label=below:{$\alpha_2$}] (A2) at (4,-2) {};
  
  % Edges
  \draw[-] (A1) -- (A3) -- (A4) -- (A5) -- (A6) -- (A7) -- (A8);
  \draw[-] (A2) -- (A4);
\end{tikzpicture}
\end{center}
The root basis $S_G = \{\gamma_1,\dots,\gamma_8\}$ can be described as
\begin{align*}
\gamma_1 &= \alpha_2 + \alpha_3 + \alpha_4 + \alpha_5 \\
\gamma_2 &= \alpha_6 + \alpha_7 \\
\gamma_3 &= \alpha_4 + \alpha_5 \\
\gamma_4 &= \alpha_1 + \alpha_3 \\
\gamma_5 &= \alpha_2 + \alpha_4 \\
\gamma_6 &= \alpha_5 + \alpha_6 \\
\gamma_7 &= \alpha_7 + \alpha_8 \\
\gamma_8 &= \alpha_3 + \alpha_4\\
\end{align*}
The volume of $W_{\flat}$ can be computed to be
\[
\prod_{\omega \in W_{\flat}} X\omega(t) = X^{64} (\gamma_1^{-18}\gamma_2^{-30}\gamma_3^{-40}\gamma_4^{-48}\gamma_5^{-54}\gamma_6^{-58}\gamma_7^{-30}\gamma_8^{-30}).
\]
The modular function for $G$ can be computed to be
\[
\delta^{-1}(t) = (\gamma_1^{14}\gamma_2^{26}\gamma_3^{36}\gamma_4^{44}\gamma_5^{50}\gamma_6^{54}\gamma_7^{28}\gamma_8^{28})(t).
\]
We get
\[
\prod_{v \in W_{\flat}} Xw(v)\delta^{-1}(t) = X^{128}\frac{Z(\beta)}{Z(b)},
\]
with $Z(\beta) = \beta_1^4 \beta_2^8 \beta_3^{10} \beta_4^{14} \beta_5^{12} \beta_6^8 \beta_7^6 \beta_8^2$.

\section{Proof of the main results}
\label{section:final}

We are finally in a position to prove Theorems \ref{theo:main} and
\ref{theorem:tail}. Before that, we present an auxiliary result bounding
the elements in $V(\Z)$ with big stabiliser.

\subsection{Congruence conditions}
\label{subs:congruence}

We want to bound the number of elements in $V(\Z)$ having a big stabiliser
in the cusp. To do that, we will apply the Selberg sieve, which in turn
requires a power saving estimate in the count of the elements in the cusp
when applying finitely many congruence conditions. 

Let $S \subset V(\Z)$ be a subset which is not necessarily $G(\Z)$-invariant. 
Analogously to Section \ref{subs:averaging}, we define
\[
N^{cusp}(S,X) = \sum_{i=1}^k\frac{1}{r_i\vol(G_0)}\int_{g \in \cF} \# \{v \in S \cap W_0 \cap (g G_0 \kappa(\Lambda L_i))_{<X}\} dg.
\]
This is the analogue of the definition of $N(S,X)$ but substituting $V(\Z)^{red}$
for the cusp $W_0(\Z)$. In the proof of Theorem \ref{theo:reducible} we saw
that
\[
N^{cusp}(W_0(\Z),X) = CX^{\dim V} + O(X^{\dim V - 1})
\]
for some constant $C$. The main theorem of this section is the following:
\begin{theorem}
\label{theo:congruence}
Let $S$ be a translate of $mV(\Z)$, for some integer $m \geq 1$. Then,
for a fixed $m$, we have that
\[
N^{cusp}(S,X) = Cm^{-\dim V}X^{\dim V} + O(m^{1-\dim V}X^{\dim V - 1}),
\]
where the implied constant is independent of $m$ and the choice of translate $S$,
as long as $m = O(X)$.
\end{theorem}
\begin{proof}
The computation is almost exactly the same as in Section \ref{subs:averaging},
with the only major difference being in the application of Davenport's lemma.
In our situation, given a bounded region $\cR$ as in the statement of
Proposition \ref{prop:Davenport}, we have that
\[
\#(\cR \cap \Z^n) = \vol(\cR) + O(\max\{\vol(\ol{\cR}),1\} ).
\]
If we now replace $\Z^n$ by a translate $L$ of $m\Z^n$, we can translate and
shrink the region $\cR$ appropriately so that $L$ gets identified with
$\Z^n$, so that Davenport's lemma yields
\[
\#(\cR \cap L) = m^{-n}\vol(\cR) + O(\max\{\vol(\ol{m^{-1}\cdot \cR}),1\} ).
\]
In our situation, what we get now instead of \eqref{eq:Davenport} is
\[
\#(S \cap (nt\cB_X)|_b) = m^{-\dim V} \vol((nt\cB_X)|_b) + O(m^{1-\dim V}X^{\dim V - 1}),
\]
where the implied constant does not change with respect to $m$ or $S$.
The hypothesis that $m = O(X)$ guarantees that none of the lower-dimensional
terms dominate. Now, the rest of the argument of Section \ref{subs:counting}
goes through in an analogous way to obtain the desired result.
\end{proof}
\begin{obs}
The added hypothesis of $m = O(X)$ is added here for convenience, and does
not affect the use of the Selberg sieve. In the notation of Section \ref{subs:Selberg},
we are really only adding the error terms for $d < D_0$, where $D_0$ is later
chosen to be a suitable power of $X$. We can always impose the additional
restriction that $D_0 = O(X)$, and the argument would go through as usual
just with a possibly worse error term. We do not do the explicit computations
of $D_0$ in this paper, but they always turn out to be $O(X)$.
\end{obs}

\subsection{Elements with big stabiliser}
Let $V^{bs}(\Z)$ be the set of elements $v \in V(\Z)$ with $\#\Stab_{G(\Q)}(v) > 1$.
Then, we are in a position to prove the following:
\begin{prop}
\label{prop:bs}
There exists a constant $\delta_{bs} > 0$ such that
\[
N^{cusp}(V^{bs}(\Z),X) = O(X^{\dim V - \delta_{bs}}).
\]
\end{prop}
\begin{proof}
By \cite[Proof of Lemma 8.22]{LagaThesis}, the density of elements
in $V(\F_p)$ having big stabiliser converges to a constant $c \in (0,1)$
as $p \to \infty$. The proof can be easily modified to show this is also
true when substituting $V(\F_p)$ by $W_0(\F_p)$.
Then, we can apply the Selberg sieve as explained Section \ref{subs:Selberg},
combined with Theorem \ref{theo:congruence}.
\end{proof}
\begin{obs}
We remark that this result depends on Theorem \ref{theo:congruence}. Namely,
to apply the Selberg sieve in that way we need a power saving estimate
on the count of reducible $G(\Z)$-orbits in $B(\Z)$, so we could not have
proven Proposition \ref{prop:bs} at the same time as Proposition \ref{prop:cusp}.
\end{obs}

\subsection{Elements with large $Q$-invariant}

In this section, we conclude the proof of Theorem \ref{theorem:tail}
about bounding elements with discriminant divisible by the square of a
large squarefree number. For $\cW_m^{(1)}$, the strongly divisible case,
it suffices to use the Ekedahl sieve as in \cite[Theorem 3.5, Lemma 3.6]{BEkedahl},
knowing that the discriminant polynomial is irreducible by \cite[Lemma 4.2]{LagaThesis}.
Thus, to conclude the proof of Theorem \ref{theorem:tail}, it suffices
to consider the weakly divisible case.

Recall that by the results in Section \ref{section:construct}, to prove 
Theorem \ref{theorem:tail} it is enough to bound the number of elements in
\[
W_M = \{v \in \frac{1}{N}V(\Z) \mid v = g\kappa_b \text{ for a squarefree } m>M,\, (m,N) = 1, \,g \in G(\Z[1/m])\setminus G(\Z), \, b \in B(\Z),\, \Delta(b) \neq 0\}.
\]
It suffices to prove that:
\begin{theorem}
There exists a constant $\delta > 0$ such that
\[
N(W_M,X) = O\lp\frac{X^{\dim V}}{M}\rp + O(X^{\dim V - \delta}).
\]
\end{theorem}
\begin{proof}
We can apply the same averaging argument as in Section \ref{subs:averaging}, up to the point where
\[
N(W_{M,i},X) = \frac{1}{r_i \vol(G_0)}\int_{n \in \omega} \int_{t \in T_c} \# \{v \in W_M \cap W_0(\Z) \cap (nt G_0 \kappa(\Lambda L_i)) \} \delta(t) dn d^{\times}t + O\lp X^{\dim V - \delta} \rp,
\]
where $W_{M,i} := W_M \cap G(\R) \kappa(\Lambda L_i)$. In light of Proposition
\ref{prop:bs}, it suffices to count elements in $W_{M,i} \cap W_0(\Z)$ with
trivial stabilizer. For each of these elements $v$, Proposition \ref{prop:p-orbit}
guarantees that $Q(v) > M$, and in particular that $|Z(v)| > M^2$. Then,
by following the same proof as in Theorem \ref{theo:reducible}, there is
some constant $C_i$ such that
\[
N(W_{M,i},X) = C_i X^{\dim V} \sum_{\substack{b \in \Z^r \\ |Z(b)| > M^2}} \frac{1}{|Z(b)|} + O(X^{\dim V - \delta}).
\]
But the written sum is $O(\frac{1}{M})$, so that concludes the proof.
\end{proof}

Therefore, we have proven Theorem \ref{theorem:tail}.
With the same proof as \cite[Theorem 4.4]{BSWsquarefree},
combining the estimates for the strongly divisible primes and the weakly
divisible primes, we get:
\begin{theorem}
For a squarefree integer $m$, let $\cW_m$ denote the elements of $B(\Z)$ 
with discriminant divisible by $m^2$. There is a constant 
$\delta > 0$ such that
\[
\sum_{\substack{m > M \\ m \text{ squarefree} \\ (m,N) = 1}} \# \{b \in \cW_m \mid \Ht(b) < X\} = O_{\eps} \lp \frac{X^{\dim V+\eps}}{\sqrt{M}} \rp + O(X^{\dim V - \delta}).
\]
\end{theorem}

\subsection{A squarefree sieve}
\label{subs:Sieve}

Theorem \ref{theo:main} follows from the previous tail estimates by performing
a squarefree sieve, following the methods in \cite[\S 4]{BSWsquarefree}.
In fact, we will prove a slightly more general result about counting elements
in $B(\Z)$ imposing infinitely many congruence conditions.

Let $\kappa$ be a positive integer. We say a subset $\cS \subset B(\Z)$ 
is \emph{$\kappa$-acceptable} if $\cS = B(\Z) \cap \bigcap_p \cS_p$,  where 
$\cS_p \subset B(\Z_p)$ satisfy the following:
\begin{enumerate}
\item $\cS_p$ is defined by congruence conditions modulo $p^{\kappa}$.
\item For all sufficiently large primes $p$, the set $\cS_p$ contains
all $b \in B(\Z_p)$ such that $p^2 \nmid \Delta(b)$.
\end{enumerate}
For any subset $A \subset B(\Z)$, denote by $N(A,X)$ the number of elements
of $A$ having height less than $X$. For any prime $p$ and any subset $A_p \subset B(\Z_p)$,
we denote by $\rho(A_p)$ the density of elements of $A_p$ inside $B(\Z_p)$.
\begin{theorem}
\label{theo:sieve}
Let $\kappa$ be a positive integer, and let $\cS \subset B(\Z)$ be a $\kappa$-acceptable
subset. Then, there exists a constant $\delta > 0$ such that
\[
N(\cS,X) = \left(\prod_{p} \rho(\cS_p) \right)N(B(\Z),X) + O(X^{\dim V - \delta }).
\]
\end{theorem}
\begin{proof}
Recall that $B = \Spec \Z[p_{d_1},\dots,p_{d_k}]$. For an element $b \in B(\Z)$ 
of height at most $X$, it holds that $|p_{d_i}(b)| < X^{d_i}$,
where by Table \ref{table:curves} we see that $d_i \geq 2$ for all $i$.
For a positive squarefree integer $m$ coprime to $N$, denote by $\cS_m'$ 
the big family defined for each prime $p$ as:
\begin{itemize}
\item If $p \mid N$, we set $\cS_p' = \cS_p$.
\item If $p \mid m$, we set $\cS_p' = B(\Z_p) \setminus \cS_p$.
\item Otherwise, we set $\cS_p' = B(\Z_p)$.
\end{itemize}
By the inclusion-exclusion principle, we get that
\[
N(\cS,X) = \sum_{\substack{m \geq 1 \\ (m,N) = 1}} \mu(m) N(\cS_m',X),
\]
where $\mu(m)$ is the Möbius function. We can estimate $N(\cS_m,X)$ as follows:
in $B(\Z)$, the set $\cS_m'$ is the union of $T_m$ translates of a congruence
class modulo $(mN)^{\kappa}$, and we have that $T_m = \prod_{p \mid m} (1 - \rho(\cS_p)) \prod_{p \mid N} \rho(\cS_p) \cdot (mN)^{k \kappa}$.
Each of these congruence classes contributes $\prod_{i=1}^k\lp \frac{2X^{d_i}}{(mN)^{\kappa}} + O(1) \rp$
to the sum $N(\cS_m,X)$. In summary, we get
\[
N(\cS_m,X) = \prod_{p \mid m} (1-\rho(\cS_p)) \prod_{p \mid N} \rho(\cS_p) N(B(\Z),X) + O(m^{\kappa}X^{\dim V - 2}).
\]
By Theorem \ref{theorem:tail}, we also have that for large enough $M$:
\[
\sum_{m \geq M} \mu(m) N(\cS_m',X) = O_{\eps} \lp \frac{X^{\dim V+\eps}}{\sqrt{M}} \rp + O(X^{\dim V - \delta})
\]
Combining the previous identities, we get
\begin{align*}
N(\cS,X) &=  \prod_{p \mid N} \rho(\cS_p) \sum_{m = 1}^M \mu(m) \prod_{p \mid m} (1-\rho(\cS_p)) N(B(\Z),X) + O_{\eps} \lp M^{\kappa+1}X^{\dim V - 2} + \frac{X^{\dim V + \eps}}{\sqrt{M}} + X^{\dim V - \delta}\rp \\
&= \prod_{p} \rho(\cS_p) N(B(\Z),X) + O_{\eps} \lp \frac{X^{\dim V}}{M}+ M^{\kappa+1}X^{\dim V - 2} + \frac{X^{\dim V + \eps}}{\sqrt{M}} + X^{\dim V - \delta}\rp,
\end{align*}
where the last estimate follows from the observation that $\rho(\cS_p) \gg 1 - \frac{1}{p^2}$
by \cite[Proof of Theorem 3.2]{PoonenSQF}. Now, optimizing we choose
$M = X^{4/(2\kappa+3)}$, which is enough for the result.
\end{proof}

\printbibliography
\end{document}